\documentclass[11pt]{amsart}

\usepackage{amsmath,leftidx,amsthm}
\usepackage[psamsfonts]{amssymb}
\usepackage[latin1]{inputenc}
\usepackage{graphicx,color}
\usepackage[all]{xy}

\usepackage{hyperref}

\usepackage[marginpar=2.5cm]{geometry}

\theoremstyle{remark}

\newtheorem*{remark}{\bf Remark}
\newtheorem{qst}{\bf Question}

\theoremstyle{plain}

\newtheorem{Theorem}{\bf Theorem}

\newtheorem{problem}{\bf Problem}

\newtheorem{theorem}{\bf Theorem}[section]
\newtheorem{proposition}[theorem]{\bf Proposition}
\newtheorem{definition}[theorem]{\bf Definition}

\newtheorem{lemma}[theorem]{\bf Lemma}

\def\A{{\mathbb A}}
\def\C{{\mathbb C}}
\def\R{{\mathbb R}}
\def\Z{{\mathbb Z}}

\def\Q{{\mathbb Q}}

\def\D{\mathbb{D}}

\def\PP{\mathbb{P}}
\def\N{{\mathbb N}}

\def\xg{x_{G}}

\def\cO{\mathcal{O}}

\def\cK{\mathcal{K}}

\def\cE{\mathcal{E}}
\def\cF{\mathcal{F}}
\def\cG{\mathcal{G}}
\def\cM{\mathcal{M}}
\def\cD{\mathcal{D}}
\def\cL{\mathcal{L}}
\def\cX{\mathcal{X}}
\def\cC{\mathcal{C}}
\def\cN{\mathcal{N}}
\def\cH{\mathcal{H}}

\def\sP{\mathsf{P}}
\def\sn{\eta}

\def\om{\omega}
\def\Om{\Omega}

\def\an{\textup{an}}

\DeclareMathOperator{\ord}{ord}
\DeclareMathOperator{\Lyap}{Lyap}

\DeclareMathOperator{\dv}{div}

\DeclareMathOperator{\mass}{Mass}

\DeclareMathOperator{\MA}{MA}

\DeclareMathOperator{\vol}{Vol}

\DeclareMathOperator{\SL}{SL}

\def\loc{\textup{loc}}

\DeclareMathOperator{\spec} {Spec}

\DeclareMathOperator{\FS} {FS}
\DeclareMathOperator{\hyb} {hyb}
\DeclareMathOperator{\NA} {NA}

\DeclareMathOperator{\red} {red}
\DeclareMathOperator{\can} {can}

\def\and{{\quad\text{and}\quad}}

\begin{document}

\title[Degeneration of complex endomorphisms]{Degeneration of endomorphisms of the complex projective space in the hybrid space}
\author{Charles Favre}
\address{CMLS, \'Ecole polytechnique, CNRS, Universit\'e Paris-Saclay, 91128 Palaiseau Cedex, France}
\email{charles.favre@polytechnique.edu}

\thanks{The author is supported by the ERC-starting grant project "Nonarcomp" no.307856, and by the brazilian project "Ci\^encia sem fronteiras" founded by the CNPq.
}

\begin{abstract}
We consider a meromorphic family of endomorphisms of degree at least $2$ of a complex projective space   that is parameterized by the unit disk.
We prove that the measure of maximal entropy of these endomorphisms converges 
to the equilibrium measure of the associated non-Archimedean dynamical system when the system  degenerates. The convergence holds
in the hybrid space constructed by Berkovich and further studied by Boucksom and Jonsson. We also infer from our analysis an estimate for the blow-up of the Lyapunov exponent 
near a pole in one-dimensional families of endomorphisms.
\end{abstract}

\maketitle

\setcounter{tocdepth}{1}
\tableofcontents

\section*{Introduction}
The main focus of this paper is to analyze degenerations of complex dynamical systems. Inspired by the recent work of S. Boucksom and M. Jonsson~\cite{boucksom-jonsson} we aim  more precisely at describing the limit of the equilibrium measures of a meromorphic family of endomorphisms of the projective space. 

More specifically we consider a holomorphic family $\{R_t\}_{t\in \D^*}$ of endomorphisms of the complex projective space $\PP^k_\C$ of degree $d\ge 2$ parameterized by the punctured unit disk, and assume it extends to a meromorphic family over $\D$. For any $t\neq 0$ small enough, one can attach to $R_t$ its unique measure of maximal entropy $\mu_t$ which is obtained as the limit $\frac1{d^{kn}} (R_t^n)^*  \om_{\FS}^{\wedge k}$ as $n\to \infty$, where $\om_{\FS}$ is the usual Fubini-Study K\"ahler form on $\PP^k_\C$.

The family $\{R_t\}$ also induces an endomorphism $\mathcal{R}$ of degree $d$ on the Berkovich analytification of the projective space $\PP^{k,\an}_{\C((t))}$ defined over the valued field of formal Laurent series endowed with the $t$-adic norm normalized by $|t|_r=r \in (0,1)$. In a similar way as in the complex case, A. Chambert-Loir~\cite{CL1} proved that the sequence of probability measures $\frac1{d^{kn}} (\mathcal{R}^n)^*  \delta_{\xg}$ converges to a measure $\mu_\mathcal{R}$
where $\xg$ is the Gau\ss{} point in $\PP^{k,\an}_{\C((t))}$. The entropy properties of $\mu_\mathcal{R}$ are much more delicate to control in this case, and $\mu_\mathcal{R}$ is no longer the measure of maximal entropy in general, see~\cite{FRL}.

We shall show that $\mu_t$ converges towards $\mu_\mathcal{R}$ as $t\to 0$.  This convergence statement is parallel to the results of L. DeMarco and X. Faber~\cite{demarco-faber} that imply the convergence of $\mu_t$ to the residual measure\footnote{This is a purely atomic measure as soon as $R_t$ diverges in the parameter space of rational maps of degree $d$.} of $\mu_\mathcal{R}$ in the analytic space $\PP^k_\C \times \D$ when $k=1$. We postpone to a subsequent article a description on how our main result yields
a higher dimensional generalization of their theorem.

To make sense out of our convergence statement we face  the difficulty  that our measures live in spaces of very different nature: complex analytic for $\mu_t$ and analytic over a non-Archimedean field for $\mu_\mathcal{R}$. 
Constructing spaces mixing both complex analytic spaces and non-Archimedean analytic spaces have appeared though  in the literature several times, most notably
in the work of J. Morgan and P. Shalen in the 80's on character varieties, see e.g.~\cite{morgan:survey}; and in a paper by V. Berkovich~\cite{berko}.
Such spaces are also implicit in the works of L. DeMarco and C. McMullen~\cite{demarco-mcmullen} and J. Kiwi~\cite{kiwi}.

We use here the construction of Berkovich which has been further clarified by Boucksom and Jonsson~\cite{boucksom-jonsson}
of a hybrid analytic space  that projects onto the unit disk $\D$ such that the preimage of $\D^*$ is naturally isomorphic to $\PP^k_\C \times \D^*$ whereas the fiber over $0$ is  homeomorphic to the non-Archimedean analytic space $\PP^{k,\an}_{\C((t))}$. 
 
The construction of this hybrid space can be done as follows. Let $A_r$ be the Banach space of complex power series
$f= \sum_{n\in \Z} a_n T^n$ such that  $|f|_{\hyb,r}:=\sum_{n\in \Z} |a_n|_{\hyb} r^n < +\infty$
where $|c|_{\hyb} = \max \{1, |c| \}$ for any $c\in \C^*$.
It was proved by J. Poineau~\cite{poineau10} that the Berkovich spectrum of $A_r$ is naturally isomorphic to the circle
$\{ |T| = r\}$ inside the affine line over the Banach ring $(\C, |\cdot|_{\hyb})$, see \S\ref{sec:hybrid circle} below. For that reason, we denote by 
$\cC_{\hyb}(r)$ this spectrum.

We may now consider the projective space  $\PP^k$ over this spectrum. To ease notation we shall write
$\PP^k_{\hyb}$ instead of $\PP^k_{\cC_{\hyb}(r)}$ (forgetting the dependence on $r$), and denote by $\pi: \PP^k_{\hyb} \to \cC_{\hyb}(r)$ the 
natural structure map.
Recall the following statement from~\cite[Appendix~A]{boucksom-jonsson}.
 \smallskip
\begin{Theorem}\label{thm:hybrid}
Fix any $r\in (0,1)$. 
\begin{enumerate}
\item
For any $t\in \bar{\D}^*_r= \{ 0< |z| \le r\}$, define $\tau(t) \in \cC_{\hyb}(r)$ by setting
$$|f(\tau(t))|:= |f(t)|^{\frac{\log r}{\log\mathopen| t\mathclose|}}$$ for all $f\in A_r$.
Then the map $t \mapsto \tau(t)$ extends to a homeomorphism from $ \bar{\D}_r$ to $\cC_{\hyb}(r)$ sending $0$ to the non-Archimedean norm $|f(\tau(0))| = |f|_r$. 
\item 
The fiber $\pi^{-1} (\tau(0))$ can be canonically identified with the Berkovich analytification $\PP^{k,\an}_{\C((t))}$ of the projective space 
defined over the field of formal Laurent series endowed with the norm $|\cdot|_r$.
\item 
There exists a unique homeomorphism $ \psi : \PP^k_\C \times\bar{\D}^*_r\to \pi^{-1} (\tau( \bar{\D}^*_r))$ such that the following holds. 
Pick any two non-zero homogeneous polynomials of the same degree 
$P_1, P_2 \in A_r[z_0, \ldots , z_k]$, and let $Q= P_1/P_2$ be the induced rational function defined  on $\PP^k_{\hyb}$.
 For any $t\in \bar{\D}_r^*$, we may evaluate the coefficients of the polynomials $P_1, P_2$, and 
we get a rational function with complex coefficients $Q_t$ outside finitely many exceptions. 
Then  we have
\[\left| Q\left(\psi([z],t)\right)\right| = \left| Q_t([z])\right|^{\frac{\log r}{\log\mathopen| z\mathclose|}}\]
for any $z$ outside the indeterminacy locus of $Q_t$.
Moreover, the composite map $\tau^{-1}\circ\pi  \circ \psi$ is equal to the projection onto the second factor. 
\end{enumerate}
\end{Theorem}

It follows from properties (2) and (3) that the construction of the hybrid space is functorial enough so that any meromorphic family of endomorphisms $\{R_t\}_{t\in \D^*}$ as above
induces a continuous (in fact analytic) map on  $\PP^{k}_{\hyb}$ whose action on $\tau^{-1}(0)$ is equal to $\mathcal{R}$.
One can now state our main result. 

\smallskip

\begin{Theorem}~\label{thm:degeneration}
Fix $r\in (0,1)$. Let $\{R_t\}_{t\in \D}$ be any meromorphic family of endomorphisms of degree $d\ge2$ of $\PP^k_\C$ that is parameterized by the unit disk, and let
$\mathcal{R}$ be the endomorphism induced by this family on the Berkovich space $\PP^{k,\an}_{\C((t))}$.
For any $t\neq 0$, denote by $\mu_t$ the measure of maximal entropy of $R_t$, 
and let $\mu_{\mathcal{R}}$ be the Chambert-Loir measure associated to $\mathcal{R}$.

Then one has the weak convergence of measures
$$
\lim_{t\to 0}\psi(\cdot,t)_* \mu_t \to  \mu_{\mathcal{R}}
$$
in $\PP^k_{\hyb}$.
\end{Theorem}

The above convergence is equivalent to the convergence of integrals
\begin{equation}\label{eq:cvg-modelfnt}
\int \Phi \, d(\psi(\cdot,t)_* \mu_t ) \longrightarrow \int \Phi \, d( \mu_{\NA} )
\text{ as } t \to 0~,
\end{equation}
for any continuous function $\Phi$ on the hybrid space. The bulk of our proof is to prove this convergence
for special functions that we call \emph{model} functions and which are defined as follows. Let $\cL$ be the pull-back of $\cO_{\PP^k_\C}(1)$ on
the product space $\PP^k_\C\times \D$, and let us fix a reference metrization $|\cdot|_{\star}$ 
on $\cL$ that we assume to be smooth and positive. 
A regular admissible datum $\cF$ is  the choice of a finite set $\tau_1, \ldots , \tau_l$ of meromorphic sections of a fixed power of $\cL$
that are holomorphic over $\PP^k_\C\times \D^*$ and have no common zeroes
 (see \S\ref{sec:admissible datum} for a precise definition). 
Any regular admissible datum gives naturally rise to a continuous function $\Phi$ on the hybrid space whose restriction to $\PP^k_\C\times \D^*$
is equal to $\log \max \{|\tau_1|_\star, \ldots , |\tau_l|_\star\}$ (see Theorem~\ref{thm:extension}). A model function is any function obtained in this way.

The key observation is that the set of model functions associated to regular admissible data forms a dense set in the space of continuous functions, see Theorem~\ref{thm:density}. 
For a model function the convergence~\eqref{eq:cvg-modelfnt} follows from direct estimates and basic facts about the definition of the complex Monge-Amp\`ere operator as defined originally by E. Bedford and A. Taylor in~\cite{bedford-taylor1}. Using refined  Chern-Levine-Nirenberg estimates due to Demailly (see also~\cite{FS-oka}), we prove that our estimates imply the convergence~\eqref{eq:cvg-modelfnt} for more general functions than model functions.  As an illustration of these ideas we analyze the behaviour of the Lyapunov exponents of $R_t$ as $t\to0$.

\smallskip

Recall that the sum of all Lyapunov exponents of the complex endomorphism $R_t$ with respect to the measure $\mu_t$ is defined by the integral
$$
\Lyap(R_t) = \int \log \mathopen\| \det (dR_t) \mathclose\| \, d\mu_t
$$
where  $\mathopen\| \det(dR_t)\mathclose\|(x)$ is the norm of the determinant of the differential of $R_t$ at a point $x$ computed in terms of e.g. the standard Fubini-Study K\"ahler metric\footnote{Since $\mu_t$ is $R_t$-invariant this quantity does not depend on the choice of the metric.} on $\PP^k_\C$.
 Observe that the integral defining the Lyapunov exponent actually converges, since $\mu_t$ is locally the Monge-Amp\`ere measure of a continuous plurisubharmonic (psh) function. 
Briend and Duval have been able to bound from below each individual Lyapunov exponent of $\mu_t$ by $\frac12 \log d$ so that  $\Lyap(R_t) \ge \frac{k}2 \log d$, see~\cite{BD-Lyap}.
Dinh and Sibony have proved in~\cite{dinh-sibony} that the function $t \mapsto \Lyap(f_t)$ is H\"older continuous and subharmonic on $\D^*$ (see also~\cite[Corollary~3.4]{BB1}).

In a non-Archimedean context one can define the quantity $\mathopen\|\det (d\mathcal{R})\mathclose\|$ using the projective/spherical metric on the projective space. 
It follows from~\cite[Th\'eor\`eme 4.1]{CLT09} (see also~\cite{MA})
that the integral $\Lyap(\mathcal{R}):=\int \log\mathopen\|\det (d\mathcal{R})\mathclose\| \, d\mu_{\mathcal{R}}$ is finite\footnote{Observe that this quantity depends on the choice of norm on $\C((t))$ hence on our given $r\in(0,1)$.}. 

To the author's knowledge, the Lyapunov exponent has been considered only in a few papers over a non-Archimedean field and just in dimension $k=1$. In this case, Okuyama~\cite{okuyama} has proved that the Lyapunov exponent can be computed as the limit of the average of the multipliers of periodic orbits of increasing periods. Jacobs~\cite{jacobs} has given some estimates of $\Lyap(\mathcal{R})$ in terms of the Lipschitz constant of $\mathcal{R}$ w.r.t. the spherical metric. Finally, the author and Rivera-Letelier~\cite{FRL16} have announced a characterization of those rational maps of $\PP^1_K$ having zero Lyapunov exponent, under the assumption that $K$ is a complete metrized field of residual characteristic zero. This applies in particular to the case $K = \C((t))$ endowed with the $t$-adic norm.

\smallskip

\begin{Theorem}~\label{thm:Lyapunov}
Under the same assumptions as in the previous theorem, we have
\begin{equation}\label{eq:lyap}
\Lyap(R_t) = \frac{\Lyap(\mathcal{R})}{\mathopen| \log r \mathclose|} \log|t|^{-1} + o (\log|t|^{-1})~.
\end{equation}
In particular we have $\Lyap(\mathcal{R}) \ge0$.
\end{Theorem}

In dimension $1$, the theorem is a consequence from works by DeMarco. More precisely, it follows from a combination of~\cite[Theorem~1.4]{dem:lyap-rat} and \cite[Proposition~3.1]{demarco16}. It can also be derived from a recent work by T. Gauthier, Y. Okuyama, and G. Vigny, see~\cite[Theorem~3.1]{approx Lyap}, on the approximation of the Lyapunov exponent by multipliers of periodic cycles. In a joint work with R. Dujardin~\cite{duj-fav}, we prove a variation of this result for meromorphic families of representations into $\SL(2,\C)$.

\smallskip

Let us mention the following
\begin{problem}\label{lyap conj}
What is the regularity of the error term
$\mathcal{E}(t):= \Lyap(R_t) - \frac{\Lyap(\mathcal{R})}{\mathopen| \log r \mathclose|} \log|t|^{-1}$
near $0$? 
Is it true that $\frac{\Lyap(\mathcal{R})}{\mathopen| \log r \mathclose|}$ is always a non-negative rational number?
\end{problem}
The rationality question of the non-Archimedean Lyapunov exponent is motivated by the work of L. DeMarco and D. Ghioca, see~\cite{demarco-ghioca}.

It follows from the plurisubharmonicity of the Lyapunov exponent that $\mathcal{E}$ is always locally bounded from above, so that $\mathcal{E}$ is actually bounded when
$\Lyap(\mathcal{R})=0$.
Even in dimension $k=1$, the error term can be however unbounded as shown by DeMarco and Okuyama in~\cite{dem:oku}.
For families of polynomials in one variable, it follows from~\cite[Corollary~1]{favre-gauthier} that $\mathcal{E}$ extends continuously at $0$,  and $\frac{\Lyap(\mathcal{R})}{\mathopen| \log r \mathclose|}$ is rational. The proof of this result is based on a former work by Ghioca and Ye~\cite{ghioca-ye} for cubic polynomials, and the result is expected to hold
true for one-dimensional algebraic families of rational maps that are defined over a number field.  

\medskip

Theorems~\ref{thm:degeneration} and~\ref{thm:Lyapunov} are consequences of results that are purely algebraic in nature and in which dynamical systems do not play any role. 
Our setup is described in \S\ref{sec:setup}, and our main results are then Theorems~\ref{thm:basic hybrid},~\ref{thm:key degeneration} and~\ref{thm:deg-uniform}.
More specifically we replace the product space $\PP^k_\C \times \D^*$ by any holomorphic family of projective varieties $X \to \D^*$ equipped with a fixed relatively ample line bundle $L \to X$. 

\medskip

We have not tried to prove our results in maximal generality. Working with families of endomorphisms of the projective space is particularly convenient since many estimates can be done relatively explicitely using homogeneous coordinates. Proving Theorems~\ref{thm:degeneration} and~\ref{thm:Lyapunov} in the context of families of polarized endomorphisms  do not require extra arguments, but we feel it would only make the reading more arduous. Note also that the projective space is the only known (to the author!) smooth variety carrying a family of polarized endomorphisms for which the dynamics is unstable and thus for which the Lyapunov exponent does not remain constant. 

 It is in any case very likely that the results presented here extend to degenerations of compact K\"ahler/hermitian manifolds in which case the hybrid space has to be replaced by the construction given in~\cite[\S 4]{boucksom-jonsson}; or even to families of varieties  with mild singularities. 
We have collected a series of questions in \S\ref{sec:questions} that we feel are of some interest for further researchs.

\bigskip

\noindent {\bf Notation}.
 $|\cdot|_0$ is the trivial norm (on any field); $|\cdot|_\infty$ is the standard euclidean norm on the field of complex numbers; and $|\cdot|_{\hyb}= \max \{|\cdot|_\infty, |\cdot|_0\}$ is the hybrid norm of Berkovich. When no confusion can arise we also write $|\cdot|= |\cdot|_\infty$ to simplify notation.

For any $r\in (0,1)$, we set $\D_r= \{z\in \C, |z| <r\}$, $\D^*_r = \D \setminus \{0\}$,  $\bar{\D}_r= \{z\in \C, |z| \le r\}$, and $\bar{\D}_r^* =  \bar{\D}_r \setminus \{0\}$.
We let $|\cdot|_r$ be the $t$-adic norm on $\C((t))$ normalized by $|t|_r =r$.
We write $\cO(\D)$ for the ring of holomorphic functions on $\D$.

\bigskip

\noindent {\bf Acknowledgements}.
I would like to extend my thanks to B. Conrad for interesting exchanges on relative ampleness in analytic geometry; and 
to S. Boucksom and M. Jonsson for all the discussions we had during our long-term collaboration on developping pluripotential tools in non-Archimedean geometry, and for their comments on a first version of this paper. I am also grateful to the referee for his very carefully reading and his constructive remarks. 


\section{Degeneration of complex projective manifolds}
In this section, we explain the construction of the hybrid space  mentioned in the introduction following Boucksom and Jonsson.
We shall work in a more general setting than strictly necessary to prove Theorems~\ref{thm:hybrid} and~\ref{thm:degeneration} so as to allow more flexibility in our arguments. 

\subsection{The general setup}\label{sec:setup}
We consider a proper submersion $\pi: X \to \D^*$ having connected fibers where $X$ is a smooth connected complex manifold of dimension $k+1$.
We suppose given an \emph{snc model} $\mathcal{X}$ of $X$ that is a smooth connected complex manifold endowed with a 
proper map\footnote{Since $\mathcal{X}$ and $\D$ are smooth and the latter is a curve, the map $\pi_{\mathcal{X}}$ is automatically flat.} $\pi_\mathcal{X}: \mathcal{X} \to \D$, together with an isomorphism $\pi_\cX^{-1}(\D^*) \simeq X$ sending $\pi_\cX$ to $\pi$
such that the central fiber $X_0:= \pi_\mathcal{X}^{-1}(0)$ is a divisor with simple normal crossing singularities. 
We get a natural embedding $\imath: X \to \mathcal{X}$.

We suppose given a line bundle $\cL \to \cX$ that is relatively very ample over $\D$. This means that the restriction of $\cL$ to the fiber $X_t= \pi_{\mathcal{X}}^{-1}(t)$ is very ample for any $t\in\D$. One can show that (up to restricting the family to a smaller
disk) this is equivalent to the existence of an embedding 
of $\cX$ into $\PP^N_\C \times \D$ compatible with $\pi_\cX$ such that $\cL$ is the restriction to $\cX$ of the pull-back of $\cO_{\PP^N_\C}(1)$ by the first projection (see~\cite[\S 1.4]{noburu1}, or~\cite[Theorem~3.2.7]{conrad}). It follows also from~\cite[Lemma 1.11]{noburu} that one may find a finite set of homogeneous polynomials $Q_1, \ldots, Q_M \in\cO(\D)[z_0, \ldots, z_N]$ in $N+1$ variables whose coefficients are holomorphic functions over $\D$ and such that $\cX = \{ ([z], t), \, Q_1(t,z) = \cdots = Q_M(t,z) = 0 \}$.

We shall denote by $t$ the holomorphic function $\pi$ on $\cX$ (with values in the unit disk). 
For any $0<r<1$ we write $\cX_r = \pi^{-1}(\D_r)$, $X_r = \pi^{-1}(\D^*_r) =\cX_r \cap X$,  $\bar{\cX}_r = \pi^{-1}(\bar{\D}_r)$, and $\bar{X}_r = \pi^{-1}(\bar{\D}^*_r)$.

\subsection{The hybrid circle}\label{sec:hybrid circle}
Fix $r\in (0,1)$. Recall from the introduction that one defines the subring  of $\C((t))$
$$A_r :=\left\{ f = \sum_{n\in \Z} a_n t^n, \, \|f\|_{\hyb,r}:= \sum_{n\in \Z} |a_n|_{\hyb} r^n < + \infty\right\}~.$$
With the norm $\|\cdot\|_{\hyb,r}$ it is a Banach ring, and we let $\cC_{\hyb}(r)$ be its Berkovich spectrum, i.e.
the set of all multiplicative semi-norms on $A_r$ that are bounded by $\|\cdot \|_{\hyb,r}$ endowed with the topology of the pointwise convergence.

Observe that for any $f\in A_r$  the set of negative integers $n$ for which $a_n \neq 0$ is finite. 
It follows  that any map $f\in A_r$ induces a continuous map on $\bar{\D}_r^*$ that is holomorphic on $\D_r$, and meromorphic at $0$.
One can thus define  a canonical map $\tau$ from the closed disk of radius $r$ to $\cC_{\hyb}(r)$ by the formulas:
\begin{equation} \label{eqn:def hyb}
\begin{cases}
|f(\tau(0))| = |f|_r = r^{\ord_0(f)}; &
\\
|f(\tau(z))| = |f(z)|^{\frac{\log r}{\log|z|}} & \text{ if } 0 < |z| \le r .
\end{cases}
\end{equation} 
for any $f\in A_r$. Observe that $|f(\tau(z))| \le \| f\|_{\hyb}$ since $\rho \mapsto \sup_{|z| =\rho} \frac{\log|f(z)|}{\log \rho}$ is non-decreasing.

This map is injective since $|(t-w) (\tau(z))| =0$ iff $z=w \neq 0$.
It is also continuous since one can write $f(t) = t^{\ord_0(f)} ( a + o(1))$ with $a\in\C^*$, and 
\[
\log |f(\tau(z))| =  \log r \, \frac{\log|f(z)|_\infty}{\log|z|_\infty} 
= \log r \,\frac{\ord_0(f) \log |z|_\infty + \log |a + o(1)|}{\log|z|_\infty}
\longrightarrow \log( r^{\ord_0(f)})~,
\]
when $z\to0$.
\begin{proposition}[\cite{poineau10}]\label{prop:hybrid circle}
The map $\tau : \bar{\D}_r \to \cC_{\hyb}(r)$ is a homeomorphism.
\end{proposition}
\begin{proof} We include a proof for the convenience of the reader. 
Since $\cC_{\hyb}(r)$ is the analytic spectrum of a Banach ring it is compact. It is therefore sufficient to prove that $\tau$ is surjective. 
Pick any multiplicative semi-norm $f \mapsto |f|$ on $A_r$ bounded by $\|\cdot\|_{\hyb,r}$.
Observe that $|t|^n = |t^n| \le r^n$ for all $n\in \Z$ which implies $|t| =r$.

Suppose first that the restriction of $|\cdot|$ to $\C$ is the trivial norm.  Then $|\cdot|$ is  non-Archimedean since
\begin{multline*}
|f+g| = |(f+g)^n|^{1/n} =\left|\sum_{i=0}^n \binom{i}{n} \, f^{i} g^{n-i}\right|^{1/n}
\le  \left(\sum_{i=0}^n  |f|^{i} |g|^{n-i}\right)^{1/n}
\le\\ (n+1)^{1/n}\, \max\{ |f|, |g|\} ~,
\end{multline*}
and letting $n\to\infty$. Pick $f = t^{\ord_0(f)} ( a + \sum_{n\ge1} a_n t^n)\in A_r$ with $a\neq0$. Then we have
\[|f| = r^{\ord_0(f)} \, \left| a + \sum_{n\ge1} a_n t^n\right| ~.\]
Since $\sum_{n\ge 0} |a_n|_{\hyb} r^n < +\infty$, we get $\lim_{N\to\infty} \sum_{n\ge N} |a_n|_{\hyb} r^n =0$, so that 
\begin{align*}
\left|\sum_{n\ge1} a_n t^n\right| \le \left|\sum_{n\ge1} a_n t^n\right|_{\hyb,r} 
&\le \max \left\{ \left|\sum_{1\le n\le N-1} a_n t^n\right|_{\hyb,r},  \left|\sum_{n\ge N} a_n t^n\right|_{\hyb,r}\right\}
\\
&\le \max \left\{|t|,  \left|\sum_{n\ge N} a_n t^n\right|_{\hyb,r}\right\} <1
\end{align*} because $|\cdot|$ is non-Archimedean.
Since $|a|=1$, it follows that $|f| = r^{\ord_0(f)}$.

\smallskip

Suppose now the restriction of $|\cdot|$ to $\C$ is non-trivial.
Then there exists a positive real number $\epsilon \le 1$ such that $|2| = |2|_\infty^\epsilon$, and this implies $|c| = |c|_\infty^\epsilon$ for any $c\in \C$.
Look at the restriction of $|\cdot|$ to the sub-algebra $\C[t]$ of $A_r$. By the Gelfand-Mazur theorem, 
this restriction has a non-trivial kernel hence $|P(t)| = |P(z)|_\infty^\epsilon$ for some $z \in \C$ and any $P\in \C[t]$. Since $|t| =r$, we have 
$|z|_\infty =r^{1/\epsilon}$. Now pick any $f\in A_r$, and expand it into power series $f(t) = t^{\ord_0(f)} (a + \sum_{n\ge 1} a_n t^n)$. 
As above we have $$\left| \sum_{n\ge N} a_n t^n\right| \le  \sum_{n\ge N} |a_n|_{\hyb} r^n \mathop{\longrightarrow}\limits^{N\to\infty} 0~,$$
and we get $|f| = \lim_{N\to\infty} \left|t^{\ord_0(f)} (a + \sum_{n\le N} a_n t^n)\right| = |f(z)|_\infty^{\epsilon}$.
\end{proof}

\subsection{The hybrid space}
Any holomorphic function $f$ on the punctured unit disk that is meromorphic at $0$ can be expanded as a series $\sum_{n\ge n_0} a_n t^n$ for some $n_0 \in \Z$
with $\sum_{n\ge n_0} |a_n| \rho^n < \infty$ for all $\rho <1$, hence belongs to $A_r$.
Since $X$ is defined by an homogeneous ideal of polynomials with coefficients in the space of holomorphic functions over $\D$,  one can make a base change and look at  the projective $A_r$-scheme $X_{A_r}$ induced by $X$. In the sequel, we fix a finite union of affine charts $U_i = \spec B_i$ with $B_i$ an $A_r$-algebra of finite type for $X_{A_r}$. If an embedding $\cX$ into $\PP^N_\C \times \D$ is fixed, then one may choose an index $i \in \{ 0, \ldots, N\}$ and look at 
$$U_i = \cX \cap \left\{ ([z],t) \in \PP^N_\C\times \D, \, z_i \neq 0\right\}$$ so that $$B_i = A_r [w_1, \ldots , w_N] / \langle Q_j (t,w_1, \ldots, w_{i-1},1,w_{i+1}, \ldots , w_N) \rangle_{j=1, \ldots , M}~.$$

\smallskip

One defines the hybrid space $X_{\hyb}$ as the analytification (in the sense of Berkovich) of  the projective scheme $X_{A_r}$ over the algebra $A_r$. 
As a topological space it is  obtained as follows.  Set $(U_i)_{\hyb}$ to be 
the space of those multiplicative semi-norms on $B_i$ whose restriction to $A_r$ is bounded by $\|\cdot\|_{\hyb,r}$, and endow this space
with the topology of the pointwise convergence.  Define $X_{\hyb}$ as the union of $(U_i)_{\hyb}$
patched together in a natural way using the patching maps defining $X_{A_r}$. 

It is a compact space (any embedding of $X$ into $\PP^N_{A_r}$ realizes $X_{\hyb}$ as a closed subset of $\PP^N_{\hyb}$, and one can check that the latter space is compact).
We get a continuous structure map $\pi_{\hyb} : X_{\hyb} \to \cC_{\hyb}(r)$ sending a semi-norm on $B_i$ to its restriction to $A_r$.

\smallskip

Observe that $A_r$ is a subring of the field of formal Laurent series $\C((t))$. 
Endowed with the $t$-adic norm $|\cdot|_r$ such that $|t|_r =r$, the field $\C((t))$ becomes complete
with valuation ring $\C[[t]]$.  We may thus consider the projective variety $X_{\C((t))}$ obtained by base change 
$A_r \to \C((t))$, and the Berkovich analytification $X^{\an}_{\C((t))}$ of $X_{\C((t))}$. 

The latter space can be defined just like the hybrid space above using affine charts, or more intrinsically as follows (see e.g.~\cite[\S 2.1]{nicaise}).
A point in $X^{\an}_{\C((t))}$ is a pair $(x,|\cdot|)$ where $x$ is a scheme-theoretic point in $X_{\C((t))}$ with residue field $\kappa(x)$ and $|\cdot| : \kappa(x) \to \R_+$
is a norm whose restriction to $\C((t))$ is $|\cdot|_r$.

The topology on $X^{\an}_{\C((t))}$ is the coarsest one such that the canonical map $\mathfrak{s} : X^{\an}_{\C((t))}\to X_{\C((t))}$ sending $(x,|\cdot|)$ to $x$ is continuous; and for any affine open set $U\subset X_{\C((t))}$, the map $f \mapsto |f(x)|$ is continuous on $\mathfrak{s}^{-1}(U)$.

Since $X_{\C((t))}$ is projective and connected, it follows from~\cite[\S 3]{berkovich} that $X^{\an}_{\C((t))}$ is  a compact locally connected and connected space. 
Note however that it is not  second countable  but is sequentially compact by~\cite{poineau12}.

\smallskip

The next result summarizes the main properties of the map $\pi_{\hyb}$. Together with Proposition~\ref{prop:hybrid circle} it also completes the proof of Theorem~\ref{thm:hybrid} from the introduction.
\smallskip
\begin{theorem}~\label{thm:basic hybrid}
\begin{enumerate}
\item
The natural map  $\pi_{\hyb} : X_{\hyb} \to \cC_{\hyb}(r)$ is continuous and proper;
\item
the central fiber $\pi_{\hyb}^{-1}(\tau (0))$ can be canonically identified with  $X^{\an}_{\C((t))}$;
\item
there exists a canonical homeomorphism $\psi: \pi^{-1}(\bar{\D}^*_r) \to \pi_{\hyb}^{-1}(\tau (\bar{\D}^*_r)) $ such that 
$ \pi_{\hyb} \circ \psi=\tau \circ \pi$, and for any rational function $\phi$ on $X_{A_r}$, one has
\begin{equation}\label{eq:identify}
|\phi(\psi(z))| = |\phi(z)|^{\frac{\log r}{\log \mathopen| z \mathclose|}}
\end{equation}
for any $z \in X$ 
outside the indeterminacy locus of $\phi$ and such that $|\pi(z)|\le r$, where $\phi$ is interpreted as a meromorphic function on $X$
in the right hand side. 
\end{enumerate}
\end{theorem}
In other words the hybrid space gives a way to \emph{see} the complex manifold $X_t = \pi^{-1}(t)$ degenerating to the non-Archimedean
analytic variety $X^{\an}_{\C((t))}$.

\begin{proof}
For the purpose of the proof, we fix a finite open cover by affine open sets $X_{A_r}= \cup U_i$ with $U_ i = \spec B_i$ where $B_i$
are $A_r$-algebras of finite type. Recall the definition of $(U_i)_{\hyb}$ which is a natural subset of $X_{\hyb}$, and let 
$(U_i)^{\an}_{\C((t))}$ be the analogous open subset of $X^{\an}_{\C((t))}$ associated to  $U_i$.

\smallskip

The continuity of $\pi_{\hyb}$ follows from the definition and the first statement is clear since $X_{\hyb}$ is compact. 

\smallskip

For the proof of (2), consider a point $x\in\pi_{\hyb}^{-1}(\tau (0))$. By definition this is a multiplicative semi-norm on some $B_i$
whose restriction to $A_r$ is equal to the $t$-adic norm $|\cdot|_r$.
It follows that $x$  naturally induces  a semi-norm on the complete tensor product $B_i \hat{\otimes} \C((t))$ (with $\C((t))$ endowed with the norm $|\cdot|_r$). 
This semi-norm is still multiplicative since the inclusion $A_r \to \C((t))$ is dense. 
We get a continuous map from $(U_i)_{\hyb}\cap \pi_{\hyb}^{-1}(\tau (0))$ to $(U_i)^{\an}_{\C((t))}$.
This map is clearly continuous, and its inverse is given by restricting a semi-norm on $B_i \hat{\otimes} \C((t))$ to $B_i$. 
These maps are compatible with the patching  procedure defining $X_{\hyb}$ and induce a canonical identification between $\pi_{\hyb}^{-1}(\tau (0))$ and $X^{\an}_{\C((t))}$. 

\smallskip

To construct the map $\psi$ in (3), recall that we realized $\cX$ as the locus $\{ ([z], t) \in \PP^N_\C \times \D, Q_j(t,z) = 0, \,  j =1, \ldots, M\}$, so that 
we may suppose
\begin{equation}\label{eq:cut affine chart}
U_i = \cX \cap \left\{ ([z],t) \in \PP^N_\C\times \D, \, z_i \neq 0\right\}
\end{equation} 
and $B_i = A_r [w_1, \ldots , w_N] / \langle Q_j (t,w_1, \ldots, w_{i-1},1,w_{i+1}, \ldots , w_N) \rangle_{j=1, \ldots , M}$.

Pick any point $([z], t) \in \PP^N_\C \times \bar{\D}_r^*$ such that $Q_j(t,z) =0$ for all $j=1, \ldots, M$.  Suppose that $z_i \neq 0$, i.e. $([z],t) \in U_i$.
We may then consider the multiplicative semi-norm 
$$f\in B_i \mapsto \left| f\left(t,\frac{z_1}{z_i}, \ldots , \frac{z_{i-1}}{z_i}, \frac{z_{i+1}}{z_i}, \ldots, \frac{z_{M}}{z_i}\right)\right|_\infty^{\frac{\log r}{\log\mathopen|t\mathclose|}}~,$$
which defines a point $\psi([z],t)\in (U_i)_{\hyb}$. This map defines a natural continuous injective map from $\pi^{-1}(\bar{\D}_r^*)$ to $\pi_{\hyb}^{-1}(\tau(\bar{\D}_r^*))$ such that  $\pi_{\hyb} \circ \psi = \tau \circ \pi$. Note also that any rational function on the $A_r$-scheme $X_{A_r}$ is the quotient of two elements in $B_i$ so that~\eqref{eq:identify}
is satisfied by definition.

Conversely pick any point $x\in \pi_{\hyb}^{-1}(\tau (\bar{\D}^*_r))$. It is a multiplicative semi-norm on some $B_i$
whose restriction to $A_r$ is equal to $\tau(t_0)$ for some $t\in \bar{\D}^*_r$ (hence to $|c|^{\frac{\log r}{\log|t|}}$ for $c\in \C$). 
The semi-norm $x$ induces a multiplicative norm on the quotient of $B_i$
by the kernel $\mathfrak{P} = \{ f \in B_i, \, |f(x)| =0\}$. This quotient is isomorphic to $\C$ by the Gelfand-Mazur theorem.
In particular, we infer the existence of a point $([z],t) \in X$ such that $|f(x)| = |f(\psi([z],t)|$ for all $f\in B_i$.

This proves $\psi$ is surjective hence a homeomorphism, and concludes the proof. \end{proof}

\begin{remark}
Observe that any semi-norm on $\cC_{\hyb}(r)$ induces a norm on $\C$ bounded by $|\cdot|_{\hyb}$ which is therefore equal to $|\cdot|_\infty^\epsilon$
for some $\epsilon\in[0,1]$. Given a point $x\in X_{\hyb}$, the restriction of $\pi_{\hyb}(x)$ to $\C$ is thus equal to $|\cdot|_\infty^{\sn(x)}$
for some normalization factor $\sn(x)\in[0,1]$. 
In this way, we get  continuous surjective and proper map $\sn: X_{\hyb}\to [0,1]$.
When $x \in X$, then it follows Theorem~\ref{thm:basic hybrid}  (3) that
$$
\sn(\psi(x) )= \frac{\log r}{\log\mathopen|\pi(x)\mathclose|^{-1}}~.
$$
\end{remark}


\section{Model functions}

We use the same setup as in the previous section. Our aim is to construct natural continuous functions (called model functions) on 
the hybrid space $X_{\hyb}$, and on the Berkovich analytic space $X^{\an}_{\C((t))}$ that are of algebraic origin and form
a dense subspace of the space of all continuous functions. These functions will play a key role in the next section to analyze degeneration of measures in $X_{\hyb}$. 

\subsection{Admissible data}\label{sec:admissible datum}
Let $\cX$ and $\cX'$ be two snc models of $X$. One shall say that $\cX'$ dominates $\cX$ if there is a proper bimeromorphic morphism $p: \cX' \to \cX$
compatible with the natural inclusion maps $\imath: X \to \cX$ and $\imath': X\to \cX'$, i.e. satisfying $\imath = p \circ \imath'$.

We say an analytic  subvariety $Z$ of an snc model $\cX$ is \emph{horizontal} when $Z$ equals the closure of $Z \cap X$ in $\cX$.
It is \emph{vertical} when it is included in the central fiber  $\cX_0$.

\begin{definition}
An admissible datum $\cF=\{\cX', d , D, \sigma_0, \ldots, \sigma_l \}$ is a collection of elements of the following form: 
\begin{itemize}
\item
an snc model $p: \cX'\to \cX$ of $X$ dominating $\cX$;
\item
 a positive integer $d\in \N^*$;
 \item
a vertical divisor $D$; 
 \item a finite set of holomorphic sections  $\sigma_0, \ldots, \sigma_l$ of the line bundle  $p^*(\cL^{\otimes d})\otimes \cO_{\cX'}(D)$  defined in a neighborhood of $p^{-1}(\bar{\cX}_r)$ whose common zero locus does not contain any irreducible component of $\cX'_0$.
  \end{itemize}
When the set of sections has no common zeroes then we say that $\cF$ is \emph{regular}.
\end{definition}
For convenience, we shall call the integer $d$ arising in the definition the \emph{degree} of the admissible datum, and refer to $D$ as the \emph{vertical divisor} associated to $\cF$.

\smallskip
There is an equivalent way of thinking about admissible data that we now explain. Recall that a (coherent) fractional ideal sheaf $\mathfrak{A}$ in $\cX$ is a (coherent) $\cO_\cX$-submodule of the sheaf of meromorphic functions such that locally $f\cdot \mathfrak{A} \subset \cO_\cX$ for some $f\in \cO_\cX$. We shall say that a fractional ideal sheaf is \emph{vertical} when its co-support (i.e. the support of the quotient sheaf $\cO_\cX/\mathfrak{A}$) is vertical. A log-resolution of a vertical fractional ideal $\mathfrak{A}$ is an snc model $p: \cX' \to \cX$ such that $\mathfrak{A}\cdot \cO_{\cX'} $ is equal to $\cO_{\cX'}(D)$ for some vertical divisor $D$. Any snc model $\cX'$ is dominated by some log-resolution of $\mathfrak{A}$ by the theorem of Hironaka on the resolution of complex analytic spaces.

\smallskip

We denote by $\dv(\sigma)$ the divisor of poles and zeroes of a meromorphic section $\sigma$ of $\cL^{\otimes d}$. Given any finite set of meromorphic sections $\sigma_i$ of $\cL^{\otimes d}$, we also let $\langle \sigma_i \rangle$ be the fractional ideal sheaf locally generated by the meromorphic functions given by $\sigma_i$ in a trivialization chart of $\cL^{\otimes d}$.

\begin{proposition}\label{prop:equiv model X}
An admissible datum is completely determined by:
\begin{enumerate}
\item
a positive integer $d\in\N^*$;
\item
a  fractional ideal sheaf $\mathfrak{A}$ in $\cX$ such that $t^N \mathfrak{A} \subset \cO_{\cX}$ for some integer $N$;
\item
a finite set of meromorphic sections $\tau_0, \ldots, \tau_l$ of $\cL^{\otimes d}$ defined in a neighborhood of $p^{-1}(\bar{\cX}_r)$ such that
$\langle\tau_0, \ldots, \tau_l\rangle = \mathfrak{A}$;
\item
an snc model $p: \cX'\to \cX$ of $X$ dominating $\cX$.
\end{enumerate}
A datum is regular iff its associated fractional ideal sheaf is vertical, and  $ p: \cX' \to \cX$ is a log-resolution of $\mathfrak{A}$.
\end{proposition}

\begin{proof}
Take any admissible datum  $\cF=\{\cX', d , D, \sigma_0, \ldots, \sigma_l \}$, and let $\sigma_{-D}$ be the canonical meromorphic section of $\cO_{\cX'}(-D)$ with 
$\dv(\sigma_{-D}) = -D$. A holomorphic section $\sigma$ of $p^*\cL^{\otimes d}\otimes \cO_{\cX'}(D)$ gives rise to a meromorphic section 
$\tau' = \sigma \sigma_{-D}$ of $p^*\cL^{\otimes d}$ which is the lift by $p$ of a meromorphic section $\tau$ of $\cL^{\otimes d}$ that is holomorphic
off $\cX_0$.

Let $\tau_0, \ldots , \tau_l$ (resp $\tau'_0, \ldots , \tau'_l$) be the meromorphic sections of $\cL^{\otimes d}$ (resp. of $p^* \cL^{\otimes d}$) associated to $\sigma_0, \ldots , \sigma_l$ as above, and set $\mathfrak{A} = \langle \tau_i \rangle$. We have 
$$\langle \tau_i \rangle\cdot\cO_{\cX'} = \langle \tau'_i \rangle  = \langle \sigma_i \rangle \cdot \cO_{\cX'}(-D)~.$$
Since $D$ is a vertical divisor, there exists an integer $N$ such that $t^N \cO_{\cX'}(-D) \subset \cO_{\cX'}$
which implies $t^N\langle \tau_i \rangle\cdot\cO_{\cX'} \subset\cO_{\cX'}$. Since any coherent ideal sheaf defined on the complement of a subvariety of codimension at least $2$ extends to a coherent ideal sheaf of the ambient variety, we get $t^N\,\mathfrak{A} \subset\cO_{\cX}$.

When $\cF$ is regular, observe that $\langle \tau_i \rangle\cdot\cO_{\cX'} = \cO_{\cX'}(-D)$ implies $\mathfrak{A}$ to be vertical, and $p$ to be a log-resolution of $\mathfrak{A}$. 

\medskip

Conversely let $\tau_0, \ldots, \tau_l$ be meromorphic sections of $\cL^{\otimes d}$ such that $t^N \mathfrak{A}\subset \cO_{\cX}$ where  $ \mathfrak{A}= \langle \tau_i \rangle$ and $N\in \N$, and pick any snc model $p: \cX' \to \cX$. Introduce the vertical divisor $D$ whose order of vanishing along an irreducible component $E$ of the central fiber is equal to $\ord_E(\mathfrak{A}) =  - \min\{\ord_E(f), \, f \in \mathfrak{A}(U)\}$ where $U$ is an affine chart intersecting $E$.
We conclude as before that  $\mathfrak{A}\cdot \cO_{\cX'}(D)$ is a coherent ideal sheaf whose co-support does not contain any irreducible component of $\cX_0$.
Any meromorphic section $\tau_i$ lifts to a meromorphic section of $p^*\cL^{\otimes d}$ whose divisor of poles and zeroes 
is greater or equal to $\dv(\mathfrak{A}\cdot \cO_{\cX'}) = -D$. In other words, the lift of $\tau_i$ to $\cX'$ is a meromorphic section $\tau'_i$
of $p^*\cL^{\otimes d}$ with $\dv(\tau'_i) \ge -D$. Since  $\langle \tau_i \rangle = \mathfrak{A}$, 
then $\langle \tau'_i \rangle\cdot \cO_{\cX'}(D)$ is a coherent ideal sheaf having horizontal co-support. 
Let $\sigma_D$ be the canonical meromorphic section of $ \cO_{\cX'}(D)$
with $\dv(\sigma_D) = D$, and define $\sigma_i = \tau'_i \sigma_D$: these are holomorphic sections of $p^*\cL^{\otimes d}\otimes \cO_{\cX'}(D)$
whose common zero locus does not contain any irreducible vertical component.

When $\mathfrak{A}$ is vertical, and $p: \cX' \to \cX$ is a log-resolution of $\mathfrak{A}$, then $\mathfrak{A}\cdot \cO_{\cX'} = \cO_{\cX'}(-D)$. 
It follows that $\langle \tau'_i \rangle\cdot \cO_{\cX'}(D) = \cO_{\cX'}$ hence the sections $\sigma_i$ have no common zeroes.
\end{proof}

\noindent {\bf Notation}.
Given any admissible datum $\cF$, we let $\mathfrak{A}_\cF$ be its associated vertical fractional ideal by the previous proposition.

\subsection{Model functions on degenerations}\label{sec:model on deg}
From now on, we fix a smooth positively curved reference metric $|\cdot|_\star$ on $\cL$ as follows. Recall that $\cX$ is embedded into $\PP^N_\C \times \D$, and $\cL$ is the restriction of the pull-back of $\cO(1)_{\PP^N_\C}$ by the first projection. Endow the ample line bundle $\cO(1)_{\PP^N_\C}$ with a metric (unique up to a scalar factor) whose  curvature form is the Fubini-Study $(1,1)$ form on $\PP^N_\C$. Pull-back this metric to the product space $\PP^N_\C\times \D$ and restrict it to $\cX$. 
\medskip

Let $\cF=\{\cX', d , D, \sigma_0, \ldots, \sigma_l \}$ be any admissible datum. Recall from the previous section that we associated to it meromorphic sections $\tau_0, \ldots , \tau_l$ of $\cL^{\otimes d}$ that are holomorphic in a neighborhood of $ \bar{X}_r = \pi^{-1}(\bar{\D}_r^*) \subset X$.

We may thus define a  function $\varphi_\cF: \bar{X}_r \to \R \cup \{ - \infty \}$ given by 
\begin{equation}\label{eq:defvarphi}
\varphi_\cF (x) :=  \log \max \{ |\tau_0|_\star, \ldots,  |\tau_l|_\star\}~.
\end{equation}
This is a continuous function on $\bar{X}_r$ with values in $\R \cup \{ - \infty \}$, and it
has finite values when $\cF$ is regular.

\begin{definition}
A model function on $X$ is a function of the form  $\varphi_\cF$ associated to a \emph{regular} admissible datum $\cF$ as above.
\end{definition}

Let $\cF$ be a (possibly singular) admissible datum.
Let $\cX''$ be any snc model dominating $\cX'$ so that the natural map $q: \cX'' \to \cX'$ is regular.
One may define an  admissible datum $\cF_{\cX''}$ by choosing the line bundle $\tilde{\cL} := q^* \hat{\cL}$ on $\cX''$, 
and considering the lift of the sections $\sigma_i$ of $\hat{\cL}$ to $\tilde{\cL}$. This new admissible datum
has the same degree as $\cF$, admits $q^*D$ as its vertical divisor, is regular when $\cF$ is, and we have
$\varphi_{\cF_{\cX''}} = \varphi_{\cF}$.

\begin{remark}
A function of the form $\varphi_{\cF}$ (e.g. any model functions) 
is completely determined by the data (1)--(3) of Proposition~\ref{prop:equiv model X}.
The snc model $\cX'$ (or equivalently the log-resolution of the fractional ideal) is included in the definition of an admissible datum 
because we shall work in such resolutions at some points.
\end{remark}

\begin{theorem}\label{thm:model1}
Denote by $\om$ the curvature of the metrization induced by $|\cdot|_\star$ on $\cL$. It is a positive closed smooth $(1,1)$ form on $\cX$
such that for any admissible datum  $\cF=\{\cX', d , D, \sigma_0, \ldots, \sigma_l \}$, the following properties hold.
\begin{itemize}
\item
In any local coordinates $(w_0, \ldots , w_k)$ such that the vertical divisor $D$ of $\cF$ is defined by the equation $\{\prod_{i=0}^k w_i^{d_i} =0\}$ with $d_i\in\N$, then we may write
\begin{equation}\label{eq:local fnt}
\varphi_\cF = \sum_{0\le i\le k} d_i \log|w_i| + v
\end{equation} where $v$ is the sum of a smooth function and a psh function with analytic singularities, so that $\varphi_\cF$ extends as an $L^1_{\loc}$-function in a neighborhood of the central fiber in $\cX'$.
\item
We have the equality of positive closed $(1,1)$-currents in $\cX'_r$:
\begin{equation}\label{eq:local struct}
dd^c \varphi_\cF    + d\, p^* \om = \Om_\cF+  [D]
\end{equation}
where $\Om_\cF$ is a positive closed $(1,1)$-current with analytic singularities. 
\item 
If $\cF$ is regular, then $v$ is continuous, and $\Om_{\cF}$ admits Lipschitz continuous potentials.
\end{itemize}
\end{theorem}

\noindent {\bf Terminology}.
We say that a positive closed $(1,1)$ current on a complex manifold has a continuous (resp. Lipschitz continuous) potential when it can be written locally
near any point as the $dd^c$ of a continuous (resp. Lipschitz continuous) psh function $u$. 
A psh function having analytic singularities is a psh function $u$ such that one can find holomorphic maps $h_0, \ldots, h_l$ and $c>0$
for which $u - c \log \max \{ |h_0|, \ldots , |h_l|\}$ is bounded. Observe that a function $u$ with analytic singularities is continuous with values in $[-\infty, +\infty)$ when it is continuous
restricted to $u^{-1}(\R)$. 

\smallskip

\begin{proof} Let $\cF$ be any admissible datum.

Choose first any point $x$ \emph{outside the central fiber}, local coordinates $(w_0, w_1, \ldots, w_k)$ near $x$, and a local trivialization of $\cL$ in that chart. In this trivialization a section 
$\sigma$  of $p^*(\cL^{\otimes d})\otimes \cO_{\cX'}(D)$ can be identified with a holomorphic function, say $h$ in the variables $w$, 
and $|\sigma|_\star = |h(w)|_\infty e^{-u}$ where $u$ is a smooth psh function. 
It follows that  $$\varphi_\cF=  \log \max\{|h_0|, \ldots , |h_l|\} - u~,$$ 
where $h_i$ are holomorphic functions.
Recall that $dd^c u$ is the curvature form of the metric $|\cdot|_\star$ on $\cL^{\otimes d}$ hence is equal to $d\, \om$, and $\log \max\{|h_0|, \ldots , |h_l|\}$ is a psh function with analytic singularities, so that~\eqref{eq:local fnt} and~\eqref{eq:local struct} hold near $x$.

\smallskip

Now choose a point $x\in \cX'_0$, and choose local coordinates $(w_0, w_1, \ldots, w_k)$
such that the central fiber $\cX'_0$ is included in $\left\{\prod_{i=0}^k w_i =0\right\}$. 
More precisely introduce the integers $a_i \in \N$, $d_i \in \Z$ such that we have the equality of divisors
$(\pi\circ p)^*[0] = \sum_i a_i [w_i =0]$; and $D = \sum_i d_i [w_i =0]$.

Choose a local trivialization of $p^*(\cL^{\otimes d})$. In this trivialization
a section $\sigma$ of  $p^*(\cL^{\otimes d})$ is a holomorphic function in the variables $w$ 
and its norm can be written as $|\sigma|_\star = |\sigma(w)|_\infty e^{-u}$ with $u$ psh  and smooth.

A section $\sigma$ of $p^*(\cL^{\otimes d})\otimes \cO_{\cX'}(D)$ can then be identified with a meromorphic function 
in the variables $w$ whose divisor of poles and zeroes $\dv(\sigma)$ satisfies  $\dv(\sigma) \ge - D$. 
In other words one can write $\sigma = \prod_i w_i^{d_i}\times  h$ where $h$ is holomorphic. 
Since $h$ is a local section of $p^*(\cL^{\otimes d})$, it follows that
$$
\log |\sigma|_\star = \sum_i d_i \log|w_i| + \log |h| -u ~,
$$
so that we may write as above
$$
\varphi_\cF = \sum_i d_i \log|w_i| +\log \max\{|h_0|, \ldots , |h_l|\} -  u~,
$$
where $u$ is a smooth psh function. Equations~\eqref{eq:local fnt} and~\eqref{eq:local struct}
follow as before.

\smallskip

When $\cF$ is regular, then the holomorphic functions $h_0, \ldots , h_l$ have no common zeroes, hence
 the function $\log \max\{|h_0|, \ldots , |h_l|\}$ is Lipschitz continuous. 
 \end{proof}

\begin{theorem}\label{thm:stability model}
The space of model functions on $X_r$  is stable by sum, and by addition by any real number.
Moreover if $\cF$ and $\cF'$ are admissible data of degree $d$ and $d'$ respectively, then 
$\frac{ \max \{ d' \varphi_\cF , d\varphi_{\cF'}\}}{\gcd{(d,d')}}$ is also a model function. 
\end{theorem}

\begin{remark}
Since $\cL$ is globally generated, one can find sections $\sigma_0, \ldots , \sigma_l $ of $\cL$  having no common zeroes, 
and  $\cF=\{\cX, 1 , 0, \sigma_0, \ldots, \sigma_l \}$ defines a regular admissible datum. In particular the space of model functions is non-empty. We shall see later
that it spans a dense subset of the space of all continuous functions on the hybrid space. 
\end{remark}

\begin{proof}
Let  $\cF=\{\cX', d , D, \sigma_0, \ldots, \sigma_l \}$ be any regular admissible datum.

Multiplying each section by a constant $\lambda\in \C^*$ modifies the model function by adding $\log|\lambda|$ to $\varphi_\cF$ which proves the stability by addition by a real number. 

\smallskip

Now pick another regular admissible datum $\cF'$. By the previous observation, and replacing $\cX'$ by a suitable snc model dominating it
we may suppose that both $\cF$ and $\cF'$ are defined over the same snc model $\cX'$. Let
$\hat{\cL}':= p^* \cL^{\otimes d'} \otimes \cO_{\cX'}(D')$ be the line bundle and 
$\sigma'_{j}$ the sections of $\hat{\cL}'$ associated to $\cF'$.

One can then build a natural regular admissible datum $\cF\otimes \cF'$ associated to $\hat{\cL} \otimes \hat{\cL}'$, and to the sections $\sigma_i \otimes \sigma'_{j}$. This new admissible datum has degree $d+d'$ and vertical divisor $D+D'$. Moreover we have 
$$\varphi_{\cF \otimes \cF'} = \varphi_{\cF} + \varphi_{\cF'}$$
which implies the stability by sum of model functions. 

\smallskip

To see the stability under taking maxima, it is easier to view the regular admissible data $\cF$ and $\cF'$ in $\cX$ given by 
their degrees $d,d'\in\N^*$, vertical fractional ideals $\mathfrak{A}, \mathfrak{A}'$ and meromorphic sections
$\tau_i$ and $\tau'_j$ of $\cL^{\otimes d}$ and $\cL^{\otimes d'}$ respectively. The log-resolutions of  $\mathfrak{A}$ and  $\mathfrak{A}'$
will not play any role in the next argument.

Introduce the integer $\delta = d d'/\gcd{(d,d')}$, and we consider the set of meromorphic sections 
$\{\sigma_i^{\delta/d}\}\cup \{(\sigma'_{i})^{\delta/d'}\}$ of $\cL^{\otimes \delta}$. The fractional ideal sheaf
$\langle \sigma_i^{\delta/d}, (\sigma'_{i})^{\delta/d'}\rangle$ is then equal to $\mathfrak{A} + \mathfrak{A}'$ which is vertical.
We may thus build an admissible datum $\cF''$ by choosing a log-resolution of $\mathfrak{A} + \mathfrak{A}'$, and
the associated model function is given by 
$\varphi_{\cF''}= \max \{ \frac{\delta}{d} \varphi_\cF , \frac{\delta}{d'} \varphi_{\cF'}\}$ as required.
\end{proof}

\subsection{Model functions on non-Archimedean analytic spaces}\label{sec:NA model}
We now explain how an admissible datum $\cF$ also induces a natural continuous function on the Berkovich analytic space 
$X^{\an}_{\C((t))}$ following the discussion of~\cite{siminag}. 

\smallskip

Let $p: \cX' \to \cX$ be any snc model of $X$ dominating $\cX$. To any irreducible component $E$ of the central fiber $\cX'_0$ we
may attach a point $x_E\in  X^{\an}_{\C((t))}$ corresponding to the generic point on the projective $\C((t))$-variety $X_{\C((t))}$ and a norm
on its field of rational functions in the following way. Pick any rational function $f$ on $X_{\C((t))}$. It defines a rational function on the $\spec \C[[t]]$-scheme  obtained by base change $\cX'_{\C[[t]]}$ whose generic fiber is isomorphic to $X_{\C((t))}$. 
We then set
\[|f(x_E)| = r^{\frac{\ord_E(f)}{b_E}}\] where $\ord_E(f)$ is the order of vanishing of $f$ at the generic point of $E$ and $b_E = \ord_E(t)\in \N^*$.

Any such point $x_E$ is called a divisorial point. It is possible to show that the set of divisorial points is dense in $X^{\an}_{\C((t))}$, see e.g.~\cite[Corollary 2.4]{siminag}.

\medskip

To any fractional ideal sheaf $\mathfrak{A}$ defined in a neighborhood of  $\bar{\cX}_r$, we can attach a function $\log|\mathfrak{A}|: X^{\an}_{\C((t))} \to \R\cup\{-\infty\}$. When 
$\mathfrak{A}$ is a vertical ideal sheaf of the $\spec \C[[t]]$-scheme $\cX_{\C[[t]]}$ this was done e.g.
in~\cite[\S~2.5]{siminag}.
Since we work here with coherent sheaves in the analytic category, we explain this construction in some details using explicit coordinates.

Recall that we realized $\cX$ as the locus $\bigcap_{j=1}^M \{ ([z], t) \in \PP^N_\C \times \D, \, Q_j(t,z) = 0\}$, 
where $Q_j$ are homogeneous polynomials in $z$ with coefficients that are holomorphic functions on $t\in \D$. 
Recall the definition of the open sets
\begin{equation*}
U_i = \cX \cap \left\{ ([z],t) \in \PP^N_\C\times \D, \, z_i \neq 0\right\}, \, i = 0, \ldots , N~.
\end{equation*}
\begin{lemma}\label{lem:global gen}
For any $i$, one can find finitely many meromorphic functions $f^{(i)}, g^{(i)}_1, \ldots, g^{(i)}_l$ defined in a neighborhood of $\bar{\cX}_r$  and  holomorphic on  $U_i$ 
such that 
$f^{(i)}\cdot \mathfrak{A}(U_i)$ is an ideal of the ring of holomorphic functions on $U_i$ that is generated by $g^{(i)}_1, \ldots , g^{(i)}_l$.
\end{lemma}
\begin{proof}
Let $H_i$ be the hyperplane section $\{ z_i =0 \}$ in $\cX$.
Recall that a section of the line bundle $\cO(dH_i)$ in a neighborhood of $\bar{\cX}_r$ for some $d$ defines
a meromorphic function in a neighborhood of $\bar{\cX}_r$ which is holomorphic on $U_i$.

It follows from a theorem of Grauert and Remmert, see~\cite{grauert-remmert} or~\cite{norguet},
that for a sufficiently large integer $d$ the sheaf $\cO(dH_i) \otimes \mathfrak{A}$ is globally generated over a neighborhood of $\bar{\cX}_r$.
This implies our claim.
\end{proof}
 
Recall that  the collection of sets  $\{(U_i)_{\hyb}\}_{i=0, \ldots, N}$ which consists of all multiplicative semi-norms on the $A_r$-algebra 
 $$B_i = A_r [w_1, \ldots , w_N] / \langle Q_j (t,w_1, \ldots, w_{i-1},1,w_{i+1}, \ldots , w_N) \rangle_{j=1, \ldots , M}$$
forms an open cover of $X^{\an}_{\C((t))}$.
For any $x\in (U_i)_{\hyb}$ we may thus set
$$
\log |\mathfrak{A}|(x) := \inf \{ \log|g^{(i)}_j(x)|,\, j = 1 , \ldots, l\} - \log |f^{(i)}(x)| ~.
$$
It is easy to check that this definition does not depend on the choice of generators, so that $
\log |\mathfrak{A}|$ actually defines a continuous function on $X^{\an}_{\C((t))}$ with values in $[-\infty, +\infty)$.
When $\mathfrak{A}$ is vertical, then the function $\log|\mathfrak{A}|$ is a real-valued continuous function. 

\smallskip

For any admissible datum $\cF$ with associated fractional ideal sheaf $\mathfrak{A}_\cF$, we set $g_\cF := \log|\mathfrak{A}_\cF|$. 
This defines a continuous function $g_\cF : X^{\an}_{\C((t))} \to [-\infty,+\infty)$ (with values in $\R$ when $\cF$ is regular).
\begin{lemma}\label{lem:formNA}
For any admissible datum $\cF$, any 
snc model $p:\cX' \to \cX$
and for any component $E$ of the central fiber, we have
$$
g_\cF(x_E) =\log r\,  \frac{\ord_E(D)}{b_E}
$$
where $D$ is the unique vertical divisor such that 
$\mathfrak{A}_\cF \cdot \cO_{\cX'} (D)$ is an ideal subsheaf of $\cO_{\cX'}$ 
whose co-support does not contain any vertical component.
\end{lemma}

\begin{proof}
We may suppose $\cX' = \cX$ and pick a generic point on $E$ which is not included in the co-support of the ideal sheaf
$\mathfrak{A}_\cF \cdot \cO_{\cX'} (D)$. In a local analytic chart $w = (w_0, \ldots, w_k)$ near that point we can write $ E = \{ w_0 =0\}$. 
The ideal $\mathfrak{A}$ is generated by $w_0^l$ for some $l\in \Z$ so that $D$ is the divisor associated to $w_0^{l}$ too (a section of $\cO_{\cX'} (D)$ is a meromorphic function whose divisor of poles and zeroes is $\ge -D = -l[w_0 =0]$).

By definition of $x_E$ we have $\log|\mathfrak{A}(x_E)| = \frac{l}{b_E}\,  \log r$ which implies our claim.
\end{proof}

\begin{definition}
A model function on the Berkovich analytic space $X^{\an}_{\C((t))}$ is a function of the form 
$g_\cF : X^{\an}_{\C((t))} \to \R$ for some regular admissible datum $\cF$.
\end{definition}

In~\cite{MA}, model functions are defined as the difference of two model functions in the sense of our paper. The notion of model functions appears at several places in the literature under various names, see~\cite[Table~1]{MA}.

\begin{proposition}\label{prop:dense NA}
Any continuous function on $X^{\an}_{\C((t))}$ is the uniform limit of a sequence of functions of the form
$q g_\cF -  q' g_{\cF'}$ where  $\cF$ and $\cF'$ are admissible data, and $q,q'$ are positive rational numbers such that $q \deg(\cF) = q' \deg(\cF')$.
\end{proposition}

\begin{proof}
Let us introduce the following three spaces of continuous functions on $X^{\an}_{\C((t))}$:
\begin{enumerate}
\item
$\cF_1 = \{ \lambda (g_\cF - g_{\cF'})\}$ where  $\cF$ and $\cF'$ are admissible data of the same degree, and $\lambda \in \Q^*_+$;
\item
$\cF_2 = \{ \lambda (\log|\mathfrak{A}| - \log|\mathfrak{B}|)\}$ where 
$\mathfrak{A}, \mathfrak{B}$ are two vertical fractional ideal sheaves defined in a neighborhood of $\bar{\cX}_r$, and $\lambda \in \Q^*_+$;
\item
$\cF_3 = \{ \lambda (\log|\hat{\mathfrak{A}}| - \log|\hat{\mathfrak{B}}|)\}$ where 
$\hat{\mathfrak{A}}, \hat{\mathfrak{B}}$ are two vertical fractional ideal sheaves of $\cX_{\C[[t]]}$, and $\lambda \in \Q^*_+$.
\end{enumerate}
By~\cite[Proposition~2.2]{siminag} we know that $\cF_3$ is dense in $\cC^0(X^{\an}_{\C((t))})$. On the other hand
Grauert and Remmert's theorem implies that for any fractional ideal sheaf $\mathfrak{A}$
there exist an integer $d\in\N^*$ and sections $\tau_0, \ldots, \tau_l$ of $\cL^{\otimes d}$ such that $\langle \tau_i \rangle = \mathfrak{A}$ over a neighborhood of $\bar{\cX}_r$.
In particular we have $\cF_1 = \cF_2$. 

To conclude it is therefore sufficient to check that for any vertical fractional ideal sheaf $\hat{\mathfrak{A}}$ of $\cX_{\C[[t]]}$ there exists a
vertical fractional  (analytic) sheaf $\mathfrak{A}$ defined in a neighborhood of $\bar{\cX}_r$ such that 
$\log|\hat{\mathfrak{A}}| =\log|\mathfrak{A}|$ on $X^{\an}_{\C((t))}$.

Replacing $\hat{\mathfrak{A}}$ by  $t^N\cdot \hat{\mathfrak{A}}$ if necessary we may suppose that  $\hat{\mathfrak{A}}$ is an coherent sheaf of ideals of 
$\cO_{\cX_{\C[[t]]}}$. Since the ideal sheaf is vertical, there exists an integer $l$ sufficiently large such that $t^l \in \hat{\mathfrak{A}}$.
Recall the definition of the open cover as in the proof of Lemma~\ref{lem:global gen}.
It follows that on 
$$(U_i)_{\C[[t]]} = \spec \C[[t]][w_1, \ldots , w_N] / \langle Q_j (t,w_1, \ldots, w_{i-1},1,w_{i+1}, \ldots , w_N) \rangle_{j=1, \ldots , M}$$
$\hat{\mathfrak{A}}$ is actually generated by elements of the ring
$$\cO(\D)[w_1, \ldots , w_N] / \langle Q_j (t,w_1, \ldots, w_{i-1},1,w_{i+1}, \ldots , w_N) \rangle_{j=1, \ldots , M}$$ 
hence by meromorphic functions on $\cX$ that are holomorphic in $U_i$. 
It thus defines a vertical ideal sheaf $\mathfrak{A}$ whose values at any point in $(U_i)_{\hyb}$ coincides with the ones of $\hat{\mathfrak{A}}$.

This concludes the proof.
\end{proof}

\subsection{Model functions on the hybrid space}
Recall that any point $x \in X_{\hyb}$ induces a norm of the field of complex numbers equal to
$|\cdot|_\infty^{\sn(x)}$ for some $\sn(x) \in [0,1]$, and we have 
$$\sn(\psi(x)) = \frac{\mathopen| \log r \mathclose|}{\log|\pi(x)|^{-1}}$$
for any $x \in X$.
\begin{theorem}\label{thm:extension}
For any admissible datum $\cF$, the function $\Phi_\cF $ given by 
\begin{equation*}
\sn \cdot \varphi_\cF \circ \psi^{-1} \text{ on } \pi_{\hyb}^{-1}(\tau(\bar{\D}^*_r)), \text{ and by } 
g_\cF \text{ on } \pi_{\hyb}^{-1}(\tau(0))
\end{equation*}
is continuous on $X_{\hyb}$ with values in $\R\cup \{-\infty\}$.
\end{theorem}
Observe that when $\cF$ is regular, then the previous result claims that $\Phi_{\cF}$ is a real-valued continuous function.
\begin{definition}
A model function on the hybrid space $X_{\hyb}$ is a continuous function of the form  $\Phi_\cF$ for some regular admissible datum $\cF$.
\end{definition}

\begin{proof}
Let $\cF$ be an admissible datum and let $\mathfrak{A}$ be its associated fractional ideal sheaf.
The continuity of $\Phi_\cF$ in restriction to $ \pi_{\hyb}^{-1}(\tau(\bar{\D}^*_r))$ (resp. to $ \pi_{\hyb}^{-1}(\tau(0))$) follows from the continuity
of $\varphi_\cF$ (resp of $g_\cF$). 

Since $ \pi_{\hyb}^{-1}(\tau(0))$ is a closed subset of $X_{\hyb}$ it is sufficient to prove the following.
For any net of points $x_n$ in  $ \pi_{\hyb}^{-1}(\tau(\bar{\D}^*_r))$ indexed by a set $\cN$ and converging to a point $x\in  \pi_{\hyb}^{-1}(\tau(0))$, then we have $\Phi_\cF(x_n) \to \Phi_\cF(x)$.

\smallskip

Pick any snc model $p: \cX'\to \cX$ obtained from $\cX$ by a sequence of blow-ups of smooth centers and write $\mathfrak{A} = \mathfrak{B} \cdot \cO_{\cX'}(-D)$
where $D$ is a vertical divisor, and $\mathfrak{B}$ is an ideal sheaf whose co-support does not contain any vertical components.
Observe that there exists a relatively ample line bundle $\cL' \to \cX'$ so that  $(\cL')^{\otimes N}\otimes \cO_{\cX'}(D) $, and $\mathfrak{B}\otimes (\cL')^{\otimes N}$ are globally generated  for a sufficiently large integer $N$ over $\bar{\D}_r$. 

We may thus find a finite family of meromorphic functions $w^{(j)}_i, h^{(j)}_\alpha$ on $p^{-1}(\pi^{-1}(\bar{\D}_r))$ such that 
$(w^{(j)}_0, \ldots , w^{(j)}_k)$ form a family of charts $U^{(j)}$ covering $\cX'_0$; 
$h_0^{(j)}, \ldots , h_l^{(j)}$ are holomorphic in $U^{(j)}$ and generate
 $\mathfrak{B}(U^{(j)})$. In each chart, we  get
\begin{equation}\label{eq:decomp}
\varphi_\cF = \sum_i d^{(j)}_i \log|w^{(j)}_i| + \log \max_\alpha \{ |h_\alpha^{(j)}|\} + \varphi^{(j)}
\end{equation} 
where $\varphi^{(j)}$ is continuous, and $D$ is defined by the equation $\left\{\prod_{i=1}^k (w^{(j)}_i)^{d^{(j)}_i} =0\right\}$.

For each $j$  let $\cN^{(j)}$ be the subset of $\cN$ of those indices $n$ such that $x_n$ belongs to the $j$-th chart. 
Write $x_n = (w_{0,n}^{(j)}, \ldots, w_{k,n}^{(j)})$ when $n\in \cN^{(j)}$.
The convergence $x_n \to x$ then implies
$\sn(x_n) \times \log|h| \to \log|h (x)|$ when $n$ is restricted to $\cN^{(j)}$, and 
for all meromorphic function $h$ on $p^{-1}(\pi^{-1}(\bar{\D}_r))$ that is holomorphic in $U^{(j)}$.
We thus get 
\begin{multline*}
\Phi_\cF(x_n) 
= 
\sn(x_n) \times \varphi_\cF \circ \psi^{-1}(x_n)
\\=
\sn(x_n) \times \left(\sum_i d^{(j)}_i \log|w^{(j)}_{i,n}| + \log \max_\alpha \{ |h_\alpha^{(j)}(x_n)|\} + \varphi^{(j)}(x_n)\right)
\end{multline*}
which implies
$$ \lim \Phi_\cF(x_n) =  \sum_i d^{(j)}_i \log|w^{(j)}_i(x)| + \log \max_\alpha \{ |h_\alpha^{(j)}(x)|\}= \log|\mathfrak{A}|(x) = g_\cF(x)~,$$
as required.
\end{proof}

\subsection{Density of model functions}

\begin{theorem}\label{thm:density}
Let $\cD(X_{\hyb})$ be the space of all functions of the form $ q \Phi_\cF - q' \Phi_{\cF'}$ where 
$\cF , \cF'$ are regular admissible data and $q, q'$ are positive rational numbers such that 
$q \deg(\cF) = q' \deg(\cF')$.

Then $\cD(X_{\hyb})$ is a $\Q$-vector space which is dense in the space of all continuous functions on $X_{\hyb}$ endowed with the topology of the uniform convergence. 
\end{theorem}

\begin{proof}
The fact that $\cD(X_{\hyb})$ is a $\Q$-vector space follows from the stability of model functions by sum, see~Theorem~\ref{thm:stability model}. 
We claim that  $\cD(X_{\hyb})$ is stable under taking maximum (hence also by minimum). 
Pick two functions  $ q_1 \Phi_{\cF_1} - q_1' \Phi_{\cF_1'}$ and $ q_2 \Phi_{\cF_2} - q_2' \Phi_{\cF_2'}$ in $\cD(X_{\hyb})$
such that $ q_1 \deg(\cF_1)= q_1' \deg(\cF_1')$, and $ q_2 \deg(\cF_2)= q_2' \deg(\cF_2')$. One can multiply both functions by a suitable large integer such that $q_1, q'_1, q_2$ and $q'_2$ are all integers. 
One then writes
$$
\max \{  q_1 \Phi_{\cF_1} - q_1' \Phi_{\cF_1'},  q_2 \Phi_{\cF_2} - q_2' \Phi_{\cF_2'}\}
=
\max \{  q_1 \Phi_{\cF_1} + q_2' \Phi_{\cF_2'},  q_2 \Phi_{\cF_2} +q_1' \Phi_{\cF_1'}  \}
 - q_1' \Phi_{\cF_1'} -   q_2' \Phi_{\cF_2'}
$$
and apply~Theorem~\ref{thm:stability model}. This proves the claim.

We then conclude by applying the Stone-Weierstrass theorem and the next Lemma.
\end{proof}

\begin{lemma}
For any two points $x\neq x'\in X_{\hyb}$, for any continuous function $g$ on $X_{\hyb}$ and any $\epsilon>0$, there exists $\Phi \in \cD(X_{\hyb})$ such that $|\Phi(x)-g(x)| \le \epsilon$ and $|\Phi(x')-g(x')| \le \epsilon$. 
\end{lemma}

\begin{proof}
Pick any two rational numbers $\rho, \rho'$, and any positive real number $\epsilon>0$. We shall prove the existence of $\Phi \in \cD(X_{\hyb})$ such that 
$|\Phi(x) - \rho|\le \epsilon$ and $|\Phi(x') -\rho'|\le \epsilon$.

If $\alpha$ is a meromorphic function on the unit disk with a single pole at $0$, we denote by $\alpha \cF$ the admissible datum
obtained by multiplying all sections by $\alpha(t)$ over $X_t$. This does not change the degree of $\cF$ but its associated vertical divisor is modified
by adding $\ord_0(\alpha)$ times the vertical divisor associated to $\cF$.  Observe that $\varphi_{\alpha \cF} = \varphi_\cF + \log|\alpha|$ on $X$, hence
\[
\Phi_{\alpha \cF} - \Phi_\cF
= \sn \cdot \left( \varphi_{\alpha \cF} \circ \psi^{-1} -  \varphi_\cF\circ \psi^{-1}\right)
= 
\sn \cdot \log\mathopen|\alpha\mathclose|\] belongs to $\cD(X_{\hyb})$.

In the case $\alpha(t) = \lambda t^q$ with $\lambda\in \C^*$ and $q\in \Z$, we get $\Phi_{\lambda t^q \cF} - \Phi_\cF = 
\sn \cdot \log|\lambda| - q \log r \in \cD(X_{\hyb})$. Since the space of model functions is stable by multiplication by any rational number, 
we get the lemma when $\sn(x) \neq \sn(x')$ (i.e.  $\mathopen|\pi_{\hyb}(x)\mathclose|\neq \mathopen|\pi_{\hyb}(x')\mathclose|$). Note that this computation also proves 
$\cD(X_{\hyb})$ contains non-zero constant functions. 

\smallskip

If $ \pi_{\hyb}(x)\neq \pi_{\hyb}(x')$ but $\sn(x) = \sn(x')$, then observe that $\sn(x) \neq 0$. We may thus find a holomorphic function  $\alpha$  on $\D$ such that 
$\sn(x) \log |\alpha(x)| = \rho$, and $\sn(x') \log |\alpha(x')| = \rho'$ which implies the lemma in this case. 

\smallskip
If $\pi_{\hyb}(x)= \pi_{\hyb}(x') =0$, i.e. both points $x,x'$ belong to $X^{\an}_{\C((t))}$, then the lemma follows from Proposition~\ref{prop:dense NA}.

\smallskip
To treat the case $x, x'$ belongs to the same fiber in $\psi^{-1}(X)$, we first recall  a few facts.
We assumed that $\cX$ is embedded in $\PP^N_\C\times\D$, and $\cL$ is the restriction to $\cX$ of the pull-back by the first projection of 
$\cO_{\PP^N_\C}(1)$. Any section $\sigma$  of $\cO_{\PP^N_\C}(d)$ is determined by a homogeneous polynomial $P_\sigma(z_0,\ldots, z_N)$ of degree $d$ in $(N+1)$-variables with complex coefficients, and by our choice of the metric on $\cL$ we have
$$|\sigma ([z])|_\star = \frac{|P_\sigma(z_0,\ldots, z_N)|}{(|z_0|^2 + \ldots + |z_N|^2)^{d/2}}~,$$
for a point $[z] = [z_0: \ldots:z_N]\in  \PP^N_\C$. A meromorphic section $\sigma$ of $\cL^{\otimes d}$ is 
therefore given by a homogeneous polynomials $P_\sigma(z_0,\ldots, z_N,t)$ of degree $d$ in $z_0, \cdots, z_N$
with coefficients depending meromorphically on $t\in \D$,
and we have
$$|\sigma (x)|_\star = \frac{|P_\sigma(z_0,\ldots, z_N,t)|}{(|z_0|^2 + \ldots + |z_N|^2)^{d/2}}~,$$
for any $x = ([z_0: \ldots:z_N],t) \in \cX\subset \PP^N_\C\times\D$.
 
 \smallskip
 
Pick $\lambda_0, \ldots , \lambda_N \in \C^*$ and integers $m_0, \ldots, m_N\in \Z$. Then 
$$
([z_0: \ldots: z_N] , t) \mapsto \log \left(\frac{\max\{|\lambda_0 t^{m_0}|\, |z_0|, \ldots, |\lambda_Nt^{m_N}|\, |z_N|\}}{(|z_0|^2 + \ldots + |z_N|^2)^{1/2}}\right)
$$
is  a model function on $X$ associated to a regular admissible datum of degree $1$ (in the snc model $\cX$, and with a non-zero vertical divisor
that depends on the choices of the integers $m_0, \ldots, m_N$).
It follows that the function  $\Phi: X_{\hyb}\to \R$ defined by   
\begin{multline*}
\Phi\left(\psi\left([z_0: \ldots:z_N],t\right)\right):= \frac{\mathopen| \log r \mathclose|}{\log|t|^{-1}}
\, \left(
\log \max\{|\lambda_0 t^{m_0}|\, |z_0|, \ldots, |\lambda_Nt^{m_N}|\, |z_N|\} \right.
- \\
\left. \log \max\{|z_0|, \ldots, |z_N|\}\right)~,
\end{multline*}
for all $([z_0: \ldots:z_N],t)\in X\subset\cX$ belongs to $\cD(X_{\hyb})$.

\smallskip

Suppose that $x \neq x' \in X$ belongs to the same fiber $X_t$ with $t\neq0$. Recall that 
the group of projective transformations of $\PP^N_\C$ preserving the Fubini-Study metric is isomorphic to the unitary group
$U(N+1)$ which acts transitively on $\PP^N_\C$. Since the metrization on $\cL$ is induced by the Fubini-Study  metrics, we may
change the embedding of $X$ by composing it by a suitable isometry, and  assume that 
 $x =([1:0:\cdots:0], t)$, and $x' =([w_0:1:\cdots:w_N], t)$ so that 
$\Phi(\psi(x)) = m_0 \log r \log |\lambda_0|$ and
\begin{multline*}
\Phi(\psi(x')) = \frac{\mathopen| \log r \mathclose|}{\log|t|^{-1}}\, \left(
\log \max\{|\lambda_0 t^{m_0}|\, |w_0|, |\lambda_1 t^{m_1}|, \ldots, |\lambda_Nt^{m_N}|\, |w_N|\}\right.
- \\\left.
\log \max\{|w_0|, 1, \ldots, |w_N|\}\right)~.
\end{multline*}
By adjusting $m_0$ and $\lambda_0$ one can achieve at $\Phi(\psi(x))$ taking its values in a fixed open interval, and 
choosing $m_1$ negative enough and $\lambda_1 =1$, $\lambda_i =0$ for all $i\ge2$, we can 
make $|\Phi(\psi(x'))|$ as large as we want.
Multiplying $\Phi$ by a suitable (small) rational number we get an element $\Phi_1 \in \cD(X_{\hyb})$ for which
$\Phi_1(\psi(x))$ and $|\Phi_1(\psi(x')) - \rho'|$ 
are both as small as we want.  In the same manner, we construct $\Phi_2 \in \cD(X_{\hyb})$ for which
$|\Phi_2(\psi(x'))|\ll 1$ and $|\Phi_2(\psi(x)) - \rho|\ll 1$, and we conclude the proof of the lemma by taking $\Phi_1 + \Phi_2$.
\end{proof}


\section{Monge-Amp\`ere measures of model functions on the hybrid space}

We now explain how a regular admissible datum $\cF$ gives rise in a natural way to a continuous family of positive measures $\mu_{t,\cF}$ on the hybrid space $X_{\hyb}$.
In \S\ref{sec:MAtocF}, we explain how to associate a continuous family of positive measures to $\cF$ on a suitable snc model. 
In \S\ref{sec:MANA}, we review briefly the definition of the Monge-Amp\`ere operator in a non-Archimedean context following~\cite{MA}, and define a measure $\mu_{\cF,\NA}$ on $X^{\an}_{\C((t))}$. In \S\ref{sec:degeneration cF} we prove the main result of this section, namely Theorem~\ref{thm:degeneration measures model} on the convergence of $\mu_{t,\cF}$ towards $\mu_{\cF,\NA}$ in the hybrid space.

\subsection{Monge-Amp\`ere measures associated to an  admissible datum}\label{sec:MAtocF}
 We refer to the survey~\cite{demailly93} for the basic theory of intersection of positive closed currents on a complex manifold. Observe that we only need the very first steps of this theory and the definition of the Monge-Amp\`ere measure of a \emph{continuous} psh function, which is due to Bedford and Taylor~\cite{bedford-taylor1}.

\smallskip

Let $\cF=\{\cX', d , D, \sigma_0, \cdots, \sigma_l \}$ be any regular admissible datum.
Recall from Theorem~\ref{thm:model1} that one can find a positive closed $(1,1)$-current
$\Om_\cF$ with Lipschitz continuous potential on the snc model $p: \cX' \to \cX$, such that
$$
\Om_\cF = dd^c \varphi_\cF - [D] + d\cdot p^* \om ~,$$
where $D$ is the vertical divisor associated to $\cF$.
Since $\Om_\cF$ has continuous potentials, its $k$-th power $\Om_\cF^{\wedge k}$ is a well-defined positive closed $(k,k)$-current on $\cX'$.
For any $t\in \D$, write $[X_t] = dd^c \log|\pi\circ p -t|$. When $t$ is non-zero, then
 $[X_t]$ is the current of integration over the fiber $\pi^{-1}(t)$, and $[X_0] = \sum b_E [E]$ where $E$ ranges over all irreducible components of $\cX'_0$ and  $b_E=\ord_E(\pi^* t)$.
 
The positive measure
$\mu_{t,\cF} = (\Om_\cF)^{\wedge k} \wedge [X_t]$ is well-defined for any $t\in\D$, and the family of measures $t \mapsto \mu_{t,\cF}$ is continuous, see e.g.~\cite[Corollary~1.6]{demailly93}.

Observe that $D$ being supported on $\cX'_0$, the measure $\mu_{t,\cF}$  for $t\in \D^*$ can be obtained alternatively 
by restricting $\Om_\cF$ to the fiber $X_t$  and consider
its Monge-Amp\`ere measure:
\begin{equation}\label{eq:defMA}
\mu_{t,\cF} = (\Om_\cF|_{X_t})^{\wedge k} =\left(d \cdot\om_t + dd^c \varphi_\cF|_{X_t}\right)^{\wedge k} 
\end{equation}
where $\om_t = \om|_{X_t}$. 
The total mass of $\mu_{t,\cF}$ is thus equal to $d^k \times \int_{X_t} \om_t^k$ which can be computed purely in cohomological terms. 
Indeed the class determined by $\om_t$ in the De Rham coholomogy  group of $X_t$ is equal to the integral class $c_1(\cL|_{X_t})\in H_{dR}^2(X_t,\Z)$, see~\cite[p.~139]{griffiths-harris}.
It follows that $$\mass( \mu_{t,\cF}) = d^k\, c_1(\cL|_{X_t})^{\wedge k} \in \N^*~.$$ Since $\mu_{t,\cF}$ varies continuously this mass is a constant.
The next computation is the key to understand the degeneration of $\mu_{t,\cF}$ as $t\to0$.

\begin{proposition}\label{prop:alg-anal}
For any irreducible component $E$ of the central fiber $\cX'_0$, one has 
\begin{equation}\label{eq:comput intersection numbers}
c_1\left(p^*\cL^{\otimes d} \otimes \cO_{\cX'}( D)|_E\right)^{\wedge k} = \int_E \Om_\cF^k\ge0~.
\end{equation}
\end{proposition}
The left hand side is computed in the DeRham (or singular) cohomology as follows: one restricts the line bundle $p^*\cL^{\otimes d} \otimes \cO_{\cX'}(D)$ to $E$, take its first Chern class, and consider the degree of  its $k$-th power. The right hand side is computed analytically, as the total mass of the measure $(\Om_\cF|_{E})^{\wedge k}$  on $E$.

\begin{proof}
Consider the line bundle $\hat{\cL} := p^* \cL^{\otimes d} \otimes \cO_{\cX'}(D)$ on $\cX'$.
A local section $\sigma$ of $\hat{\cL}$ is the same as a local section of $p^* \cL^{\otimes d}$ whose divisor of poles and zeroes satisfies $\dv(\sigma) \ge -D$.
Endow $\hat{\cL}$ with the metric $|\cdot|_{\cF} := |\cdot|_\star\, e^{-\varphi_\cF}$. Choose coordinates $w$ in a trivializing chart such that $D$ is given by the equation $\{ \prod_i w_i^{d_i} =0\}$. By  Theorem~\ref{thm:model1}, we have
$|\sigma|_{\cF} = |\sigma(w)|_\infty e^{-u} e^{-g_\cF}$ with $u$ smooth and $dd^c u = d p^*\om$. 
Since $\varphi_\cF = \sum d_i \log |w_i| + v$ with $v$ smooth  we see that $w \mapsto |\sigma(w)|_\cF = e^{v-u}\,  |\sigma(w)|_\infty \prod_i |w_i|^{-d_i}$ is continuous.  It follows that $|\cdot|_{\cF}$ is a continuous metric on $\hat{\cL}$ whose curvature form is equal to $\Omega_\cF$ by~\eqref{eq:local struct}. Therefore $c_1(\hat{\cL}|_E)$ is represented by the positive closed $(1,1)$-current $\Omega_\cF|_E$ and the formula follows from~\cite[Corollary~9.3]{demailly93}.
\end{proof}

\subsection{Monge-Amp\`ere measures on $X^{\mathrm{an}}_{\C((t))}$}\label{sec:MANA}
We briefly review  the definition of the Monge-Amp\`ere operator following A. Chambert-loir~\cite{CL1,CL2}. 
The theory has been expanded and made more precise in~\cite{MA,nama-survey},~\cite{GM}, and we shall extract from the first reference 
the key Theorem~\ref{thm:key result} below.

\smallskip

Recall that  $X_{\C((t))}$ is the projective variety over the field $\C((t))$ obtained from $X$ by base change $A_r \to \C((t))$. 
We shall also consider the $\spec(\C[[t]])$-scheme  $\cX_{\C[[t]]}$ obtained by base change $A_r \to \C[[t]]$. It is a formal scheme whose
generic fiber is $X_{\C((t))}$. We also denote by $X^{\an}_{\C((t))}$  the Berkovich analytification of $X_{\C((t))}$ when $\C((t))$ is endowed with the $t$-adic norm 
with $|t|_r =r$.

The line bundle $\cL \to \cX$ induces natural line bundles $L_{\C((t))} \to X_{\C((t))}$,  $L_{\C[[t]]} \to\cX_{\C[[t]]}$, and
$L^{\an}_{\C((t))} \to X^{\an}_{\C((t))}$. Recall that $L_{\C[[t]]}$ determines a natural  
metrization $|\cdot|_{\cL}$ on $L^{\an}_{\C((t))}$, see~\cite[\S 1.3.2]{CL2}.
Any other continuous metrization $|\cdot|$ on $L^{\an}_{\C((t))}$ can be thus written $|\cdot| = |\cdot|_\cL e^{-g}$ for some continuous function $g: X^{\an}_{\C((t))} \to \R$. 

\smallskip

We shall say that $|\cdot|$ is a semi-positive model metrization if $g$ is a positive rational multiple of a model function $\log|\mathfrak{A}|$, and for some (or any) log-resolution $p: \cX' \to \cX$ of $\mathfrak{A}$ such that $\mathfrak{A}\cdot \cO_{\cX'} = \cO_{\cX'}(-D)$, the line bundle $p^* \cL \otimes \cO_{\cX'}(D)$ is relatively nef in the sense that 
$p^* \cL \otimes \cO_{\cX'}(D)|_E$ is nef for all irreducible component $E$ of $\cX'_0$.

\medskip

To any semi-positive model metrics $|\cdot| = |\cdot |_{\cL} e^{-g}$ as above, we associate a positive (atomic) measure\footnote{Chambert-Loir uses the notation $( \widehat{c_1}(\overline{L}_g)^k| X)$ instead of 
$\MA_{\cL}(g)$. The latter notation is inspired by the notations used in~\cite[\S 4]{MA}.}  $\MA_{\cL}(g)$ on $X^{\an}_{\C((t))}$ as follows:
\begin{equation}\label{eq:defi-MANA}
\MA_{\cL}(g) := \sum_E c_1\left(p^* \cL \otimes \cO_{\cX'}(D)|_E\right)^{\wedge k} \, \delta_{x_E}
\end{equation}
where $E$ ranges over all irreducible components of the central fiber $\cX'_0$, and $x_E$ is the divisorial point associated to $E$ as in \S \ref{sec:NA model}.

The quantity $c_1\left(p^* \cL \otimes \cO_{\cX'}(D)|_E\right)^{\wedge k}$ is understood as follows. We restrict the line bundle $\hat{L} := p^* L_{\C[[t]]} \otimes \cO_{\cX'}(D)$ to $E$ viewed as a component of the special fiber of the formal scheme $X_{\C[[t]]}$, and compute the top intersection degree of its first Chern class $c_1(\hat{L}|_E)$ (in $E$ viewed as a projective $\C$-scheme). Since the (complex) analytification of $\hat{L}|_E$ is isomorphic to  $p^* \cL \otimes \cO_{\cX'}(D)|_E$,  we see that $c_1\left(p^* \cL \otimes \cO_{\cX'}(D)|_E\right)^{\wedge k}$ is identical to the left hand side of~\eqref{eq:comput intersection numbers} by the compatibility results of~\cite[Example 19.1.1 \& Corollary 19.2 (b)]{fulton}.

\medskip

A general continuous semi-positive metric $|\cdot| = |\cdot|_\cL e^{-g}$ is by definition a continuous metric on $L^{\an}_{\C((t))}$ such that there exists a sequence of semi-positive model metrics $ |\cdot|_n = |\cdot|_\cL e^{-g_n}$ for which $g_n \to g$. One associates to any such metric a positive Borel measure on $X^{\an}_{\C((t))}$ by setting
$\MA_{\cL}(g) = \lim_{n\to\infty} \MA_{\cL}(g_n)$. This measure does not depend on the choice of model metrics converging to $|\cdot|$.

A (singular) semi-positive metric $|\cdot|_{\cL}e^{-g}$ is by definition determined by an upper semi-continuous function $g : X^{\an}_{\C((t))} \to [-\infty, +\infty)$ for which there exists a net of model semi-positive metrics $|\cdot|_\cL e^{-g_n}$ such that $g_n$ is decreasing pointwise to $g$, see~\cite[Theorem~B]{siminag}. 

In this terminology, we have the following result.
\begin{theorem}\label{thm:key result}
Let $|\cdot|_{\cL}e^{-g_n}$ be a sequence of continuous semi-positive metrics on $L^{\an}_{\C((t))}$ converging uniformly to $|\cdot|_{\cL}e^{-g}$. Then the latter metric is again a continuous semi-positive metric and we have
$$
\MA_{\cL}(g) = \lim_{n\to\infty} \MA_{\cL}(g_n)~.
$$
More precisely, given any singular semi-positive metric $|\cdot|_{\cL} e^{-h}$, all integrals $\int h\, d(\MA_{\cL}(g_n))$ and $\int h\, d(\MA_{\cL}(g))$ are finite, and we have
\begin{equation}\label{eq:347}
\int h\, d(\MA_{\cL}(g)) = \lim_{n\to\infty}\int h\,  d(\MA_{\cL}(g_n))~.
\end{equation}
\end{theorem}

\begin{proof}
This result is essentially due to Chambert-Loir and Thuillier, see~\cite[Th\'eor\`eme 4.1]{CLT09}. Since we followed notations and conventions from~\cite{MA} we sketch a proof
following the latter reference.  The fact that $|\cdot|_{\cL}e^{-g}$ is semi-positive and the weak convergence is a direct consequence of~\cite[Theorem~3.1]{MA}.
The finiteness of the integrals is exactly~\cite[Proposition~3.11]{MA}. To prove~\eqref{eq:347}, we freely use notation from~\cite{MA}.

Let $\theta$ be the class in the relative Neron-Severi space $N^1(\cX_{\C[[t]]}/S)$ induced by $c_1(\cL)$ (with $S = \spec \C[[t]]$). A $\theta$-psh function $g$
is an upper semi-continuous function whose metric  $ |\cdot|_\cL e^{-g}$ is semi-positive. For any continuous $\theta$-psh functions $g_1, \ldots, g_k$, ~\cite[Theorem~3.1]{MA} asserts that one can define
a Radon measure $(\theta + dd^c g_1)\wedge \cdots \wedge  (\theta + dd^c g_{k})$. This measure has mass $\delta = \int \theta^k = c_1(L)^k$, is symmetric in the entries, and depends continuously on the $g_i$'s. When all functions are the same $g = g_1 = \ldots = g_k$, then we have $\MA_{\cL}(g) = (\theta + dd^c g)^{\wedge k}$. 

\smallskip

The first step is to prove that one can define a (signed) Radon measure $ dd^c h  \wedge (\theta + dd^c g_1)\wedge \cdots \wedge  (\theta + dd^c g_{k-1})$ when $h$ is any $\theta$-psh function, and the $g_i$'s are continuous $\theta$-psh functions. 
The point is to check that for any model function $\varphi$, the quantity 
$$
\Lambda(\varphi) := \int  h \, (dd^c \varphi) \wedge (\theta + dd^c g_1)\wedge \cdots \wedge  (\theta + dd^c g_{k-1})$$
is well-defined and satisfies $|L(\varphi)| \le 2\delta\, \sup|\varphi|$.
To see that, one first assumes that $h$ is bounded and one writes
\begin{multline*}
\pm\, \int  h \, (dd^c \varphi) \wedge (\theta + dd^c g_1)\wedge \cdots \wedge  (\theta + dd^c g_{k-1})
= \\
\pm \, \int   \varphi \, (dd^c h) \wedge (\theta + dd^c g_1)\wedge \cdots \wedge  (\theta + dd^c g_{k-1})
\le 2 \sup |\varphi| \,\delta~.
\end{multline*}
For a general $h$, we apply the very same estimate to the sequence $\max \{ h , -n\}$ and let $n\to\infty$.

Since the linear form $\varphi \mapsto \Lambda(\varphi)$ is continuous, it defines a Radon measure on $X^{\an}_{\C((t))}$ (of total mass $\le 2 \delta$) 
which we denote by  $dd^c \psi \wedge (\theta + dd^c g_1)\wedge \cdots \wedge  (\theta + dd^c g_{k-1})$. 
Then we write: 
\begin{multline*}
\int \psi\, d(\MA_{\cL}(g)) - \int \psi\, d(\MA_{\cL}(g_n))
=
\int \psi\, (\theta + dd^c g)^{\wedge k} - \int \psi\, (\theta + dd^c g_n)^{\wedge k}
\\
=
\sum_{j=0}^{k-1}
\int \psi\, dd^c (g- g_n) \wedge (\theta + dd^c g)^j \wedge (\theta + dd^c g_n)^{k-j-1} 
\\
=
\sum_{j=0}^{k-1}
\int (g- g_n)\, dd^c(\psi)  \wedge (\theta + dd^c g)^j \wedge (\theta + dd^c g_n)^{k-j-1} 
\le 
2k\, \sup|g-g_n| \, \delta~,
\end{multline*}
which concludes the proof.
\end{proof}

In the sequel we shall use the following  computation.
\begin{proposition}\label{prop:comput MANA}
Let  $\cF=\{\cX', d , D, \sigma_0, \ldots, \sigma_l \}$ be any regular admissible datum.
Then the metric $|\cdot|_{\cL} e^{-g_\cF}$ is a semi-positive model metric, and for
any (possibly singular) admissible datum $\cG$, we have
\begin{equation}\label{eq:fndt MA}
\int g_{\cG} \, \MA_{\cL}(g_\cF) =\sum_E g_\cG(x_E)\, c_1\left(p^*\cL^{\otimes d} \otimes \cO_{\cX'}(D)|_E\right)^{\wedge k}
\end{equation}
where $E$ ranges over all irreducible components of $\cX'_0$.
\end{proposition}
\begin{proof}
Since the sections $\sigma_0, \ldots, \sigma_l $
of  the line bundle $\hat{\cL}:= p^* \cL^{\otimes d} \otimes \cO_{\cX'}(D)$ 
have no common zeroes over $\bar{\D}_r$, for any compact curve $C\subset \cX'_0$ there exists at least one section say $\sigma_0$ whose restriction to $C$ is non-zero and
$$
\deg(\hat{\cL}|_C) = \sum_{p\in E} \ord_p(\sigma_0|_C) \ge 0
 $$
so that $\hat{\cL}$ is relatively nef. This proves $|\cdot|_{\cL} e^{-g_\cF}$ is a semi-positive model metric. 
The identity~\eqref{eq:fndt MA} then follows from the definition of $\MA_{\cL}(g_\cF)$ when computed in $\cX'$.
\end{proof}

\subsection{The Chambert-Loir measure associated to an endomorphism of $\PP^{k,\mathrm{an}}_{\C((t))}$}\label{sec:ACL measure}
This section may be skipped on  a first reading. Suppose $\cX= \PP^k_\C \times \D$, and let $\cL$ be the pull-back by the second projection of $\cO_{\PP^k_\C}(1)$.
This line bundle determines a canonical semi-positive metric $|\cdot|_{\can}$ on $\cO_{\PP^{k,\an}_{\C((t))}}\!(1)$, and we shall also denote by $|\cdot|_{\can}$ 
the induced metric on $\cO_{\PP^{k,\an}_{\C((t))}}\! (d)$ for all $d\in \Z$.
Note that the norm of a section $\sigma$ of $\cO_{\PP^{k,\an}_{\C((t))}}\! (d)$ is given in homogeneous coordinates by
\[
|\sigma([w])|_{\can} = \frac{|P_\sigma(w_0, \ldots, w_k)|}{\max\{|w_0|^d, \ldots, |w_k|^d\}}~,
\]
where $P_\sigma$ is the homogeneous polynomial (of degree $d$ and coefficients in $\C((t))$) determined by $\sigma$. The Monge-Amp\`ere measure of $|\cdot|_{\can}$ is the Dirac mass at the divisorial point\footnote{When suitably interpreted as a norm on $\C((t))[z_1,\ldots, z_k]$ this point corresponds to the Gau{\ss} norm hence the notation, see~\cite[\S 2.1]{CL2}.} $\xg$ corresponding to $\PP_\C^k\times\{0\}$. In the notation of the previous section, we thus have $\MA_{\cL}(0) = \delta_{\xg}$.

\medskip

Now suppose $\mathcal{R}$ is an endomorphism of $\PP^k_{\C((t))}$ of degree $d$ given in homogeneous coordinates by $k+1$ polynomials $P_0, \cdots , P_k \in \C((t)) [w_0, \cdots, w_k]$ of degree $d$ having no zeroes in common except for the origin.

There is a natural way to pull-back metrics by regular maps. 
Observe that  the pull-back metric $\mathcal{R}^*|\cdot|_{\can}$ on 
$\mathcal{R}^*\cO_{\PP^{k,\an}_{\C((t))}}\!(1) = \cO_{\PP^{k,\an}_{\C((t))}}\!\!(d)$ can be written $\mathcal{R}^*|\cdot|_{\can} = |\cdot|_{\can}e^{-g_1}$
where
\[
g_1([w]) = \log \left(\frac{\max\{|P_0|, \ldots, |P_k|\}}{\max\{|w_0|, \ldots, |w_k|\}^d}\right)~.
\]
The metric $\mathcal{R}^*|\cdot|_{\can}$ is again semi-positive, see e.g.~\cite[Lemma~2.10]{favre-gauthier} for details.  Consider now the metric $|\cdot|_n$ on $\cO_{\PP^{k,\an}_{\C((t))}}\!(1)$  obtained by taking the 
 $d^n$-th root of $(\mathcal{R}^n)^*|\cdot|_{\can}$. We get \[|\cdot|_{n+1} = |\cdot|_n e^{- \frac1{d^n} g_1 \circ \mathcal{R}^n}\] so that  $|\cdot|_{n+1}$ converges uniformly to a continuous semi-positive metric $|\cdot|_{\mathcal{R}} = |\cdot|_{\can} e^{-g_{\mathcal{R}}}$ on $\cO_{\PP^{k,\an}_{\C((t))}}\!\!(1)$ with
 \begin{equation}\label{eq:gR}
g_{\mathcal{R}} = \sum_{n\ge 0}\frac1{d^n} \, g_1 \circ \mathcal{R}^n~.
 \end{equation}
 The Chambert-Loir measure associated to $\mathcal{R}$ is by definition $\mu_{\mathcal{R}} := \MA_{\cL}(g_{\mathcal{R}})$.

\subsection{Degeneration of measures}\label{sec:degeneration cF}
Let us return to our general setup as described in \S\ref{sec:setup}.

Fix a regular admissible datum $\cF$.  Recall the definition of $\psi$ and the inclusion of the Berkovich analytification $X^{\an}_{\C((t))}$ into the hybrid space given by Theorem~\ref{thm:basic hybrid}. Using Proposition~\ref{prop:comput MANA},  $\cF$ defines a positive on $X^{\an}_{\C((t))}$. 
On the other hand, we also have a family of complex Monge-Amp\`ere measures on $X$, see~\eqref{eq:defMA} so that
 we may define a family of positive measures $\mu_{t,\cF,\hyb}$  on $X_{\hyb}$ parameterized by $t\in\bar{\D}_r$ by setting:
\[\begin{cases}
\mu_{t,\cF,\hyb} :=  \psi_* (\mu_{\cF,t}) \text{ if } t\in \bar{\D}_r^*~;\\ 
\mu_{0,\cF,\hyb}:= \MA_{\cL}(g_\cF)~.
\end{cases}\]
Observe that $\mu_{t,\cF,\hyb}$ is supported on $\pi_{\hyb}^{-1}(\tau(t))$.
We then have the following continuity statement.
\begin{theorem}\label{thm:degeneration measures model}
For any regular admissible datum $\cF$, one has the weak convergence of measures in $X_{\hyb}$: 
$$
\lim_{t\to0} \mu_{t,\cF,\hyb} = \mu_{0,\cF,\hyb}~.
$$
\end{theorem}
By the Density Theorem~\ref{thm:density}, this continuity statement  follows from $$\lim_{t\to0} \int \Phi_{\cG}\,d\mu_{t,\cF,\hyb} = \int \Phi_{\cG}\,d\mu_{0,\cF,\hyb}$$ for any regular admissible datum $\cG$. Since we have $\int \Phi_{\cG}\,d\mu_{0,\cF,\hyb} = \int g_{\cG} \, d\mu_{0,\cF,\hyb}$ by definition, we see that the continuity is in fact a consequence of the following (more general) statement by~\eqref{eq:fndt MA}.
\begin{theorem}\label{thm:key degeneration}
Let $\cF$ and $\cG$ be admissible data, with $\cF$ regular.
Then one has 
$$
\lim_{t\to0} \int \Phi_{\cG}\,d\mu_{t,\cF,\hyb} =  \sum_Eg_{\cG} (x_E)\, c_1\left(p^*\cL^{\otimes d} \otimes \cO_{\cX'}(-D)|_E\right)^{\wedge k}~,
$$
where the sum is taken over all irreducible components $E$ of the central fiber of an snc model $\cX'$ which is a log-resolution of the fractional ideal sheaf associated to $\cF$. 
\end{theorem}
\begin{proof}
Choose any snc model $p: \cX'\to \cX$ which is a log-resolution of the vertical fractional ideal sheaf
associated to $\cF$. Decompose the fractional ideal sheaf $\mathfrak{A}$ associated to $\cG$ by writing
$\mathfrak{A} = \mathfrak{B} \cdot \cO_{\cX'}(-D)$ where $D$ is a vertical divisor, and $\mathfrak{B}$ is an ideal sheaf whose co-support $W$
does not contain any vertical component. We shall denote by $Z$ the union of $\cX'_0\cap W$ and the singular locus of the central fiber $\cX'_0$: it is a subvariety included in $\cX'_0$ that does not contain any irreducible component of the central fiber.

\smallskip

Cover the central fiber $\cX'_0$ by finitely many charts $U^{(j)}$ and choose coordinates $w^{(j)}= (w^{(j)}_0, \ldots , w^{(j)}_k)$ in each of these charts. 
Let $I_j \subset \{ 0, \ldots ,k\}$ be the  subset of indices for which $\{w^{(j)}_i =0 \}$ is included in the central fiber, and
let $b^{(j)}\in\N^*$ be such that  one has $$t = \pi \circ p = \prod_{i\in I_j} (w^{(j)}_i)^{b^{(j)}_i} \times \text{unit}$$ in $U^{(j)}\subset \cX'$.
By Theorem~\ref{thm:model1}, one can also find integers $d^{(j)}_{i,\cF}$ and  $d^{(j)}_{i,\cG}$, and finitely many holomorphic functions $h_\alpha^{(j)}$ such that 
$\mathfrak{B} \cdot \cO_{\cX'}(U^{(j)}) = \langle h_\alpha^{(j)}\rangle$,
$$\varphi_\cF = \sum_i d^{(j)}_{i,\cF} \log|w^{(j)}_i| + \varphi^{(j)}, \text{ and } \varphi_\cG = \sum_i d^{(j)}_{i,\cG} \log|w^{(j)}_i| +  \log \max_\alpha  |h_\alpha^{(j)}| + \psi^{(j)}$$ on $U^{(j)} $ where $\varphi^{(j)}$ and $\psi^{(j)}$ are continuous. 
It follows that one can write for all  $w^{(j)}\in U^{(j)}\setminus \cX'_0$:
\begin{align}
  \Phi_{\cG} \circ \psi(w^{(j)})  &= \frac{\mathopen| \log r \mathclose|\cdot \varphi_\cG (w^{(j)})}{\log\left|\pi\circ p(w^{(j)})\right|^{-1}} \nonumber \\
    &=\frac{\log r\cdot  \left(\sum_{i \in I_j}  d^{(j)}_{i,\cG} \log\left|w^{(j)}_i\right| +\log \max_\alpha  \left|h_\alpha^{(j)}\right|\right) + O(1) }{\sum_{i \in I_j} b^{(j)}_i \log\left|w^{(j)}_i\right| + O(1)}~.\label{eq:estim-important}
\end{align}
Let $K$ be any compact neighborhood of  $Z$ inside $\cX'$.
Observe that  all integers $ b^{(j)}_i$ are non-zero, that $\max_\alpha \left|h_\alpha^{(j)}\right|$ is bounded from below outside $K$, and that 
$w^{(j)}_l$ is bounded from below too if $w^{(j)}_i\to0$ since $Z$ contains the singular locus of the central fiber. 
It follows that 
\begin{equation}\label{eq:estim crucial}
\Phi_{\cG} \circ \psi(w^{(j)}) \to \frac{d^{(j)}_{i,\cG}}{b^{(j)}_i}\,\log r \text{ when } w^{(j)}_i \to 0 \text{ and } w^{(j)} \notin K~. 
\end{equation}
In more geometric terms, these estimates imply the
\begin{lemma}\label{lem:1s}
The function $\Phi_{\cG} \circ \psi$ extends to a continuous function on $\cX'\setminus K$
whose restriction to an irreducible component $E$ of the central fiber is constant equal to $g_{\cG}(x_E)$.
\end{lemma}
\begin{proof}
The equation~\eqref{eq:estim crucial} implies the continuity statement. Let $E$ be an irreducible component of $\cX'_0$, and suppose $U^{(j)} \cap E$ is non empty
and determined by the equation $w^{(j)}_i =0$. Then by Lemma~\ref{lem:formNA} we have
$$g_{\cG}(x_E)
= \log r\, \frac{\ord_E(D)}{b_E}
$$
where $D$ is the vertical divisor associated to $\cG$, and $b_E = \ord_E(\pi \circ p)$. It follows from Theorem~\ref{thm:model1} that 
$D$ is given by the equation  $(w^{(j)}_i)^{d^{(j)}_i} =0$ in $U^{(j)}$ whereas $b_E = b^{(j)}_i$. This concludes the proof.
\end{proof}
We shall also use the following
\begin{lemma}\label{lem:2s}
For any $\epsilon>0$, there exists a compact neighborhood $K$ of $Z$, such that 
\begin{equation}\label{eq:fkbound}
\max \left\{\int_{K} (\Om_{\cF}|_{X_t})^{\wedge k}, \int_{K} \left| \Phi_{\cG} \circ \psi\right|\,  (\Om_{\cF}|_{X_t})^{\wedge k}\right\}\le \epsilon
\end{equation} for any $t\in\D$.
\end{lemma}
To simplify notation, write $\mu_E = c_1\left(p^*\cL^{\otimes d} \otimes \cO_{\cX'}(D)|_E\right)^{\wedge k}$ for any irreducible component $E$ of $\cX'_0$.
We then obtain
\begin{multline*}
\Delta_t := \left| \int \Phi_{\cG} \, d\mu_{\cF,t\hyb} -  \sum_E g_{\cG}(x_E) \mu_E \right|
= 
\left| \int \left( \Phi_{\cG}\circ \psi\right) \,d\mu_{\cF,t} -  \sum_E g_{\cG}(x_E) \mu_E \right|
\le 
\\
\left|\int_K \left( \Phi_{\cG}\circ \psi \right)\, (\Om_{\cF}|_{X_t})^{\wedge k}\right|
+ 
\left| \int_{X_t\setminus K} \left( \Phi_{\cG}\circ \psi \right) \, (\Om_{\cF}|_{X_t})^{\wedge k} -  \sum_E g_{\cG}(x_E) \mu_E \right|~.
\end{multline*}
Applying~\eqref{eq:fkbound} and Lemma~\ref{lem:1s}, we get 
$$
\varlimsup_{t\to0} \Delta_t 
\le \epsilon+  \sum_E \left| g_{\cG}(x_E)\right| \left(\int_{E\setminus K} (\Om_{\cF}|_E)^{\wedge k}- \mu_E\right)~.$$
By Proposition~\ref{prop:alg-anal}, we have $\int_{E} (\Om_{\cF}|_E)^{\wedge k}= \mu_E$ so that 
$$\varlimsup_{t\to0} \Delta_t  \le \epsilon + \sum_E \left| g_{\cG}(x_E)\right| \left(\int_{E\cap K} (\Om_{\cF}|_E)^{\wedge k}\right)~.$$
We now apply Lemma~\ref{lem:2s} and choose a compact set $K$ such that all integrals $\int_{E\cap K} (\Om_{\cF}|_E)^{\wedge k}$
are $\le \epsilon$. We conclude that $\varlimsup_{t\to0} \Delta_t \le \epsilon (1 + \sup g_{\cG})$ which can be made arbitrarily small. 
This concludes the proof of Theorem~\ref{thm:key degeneration}.
\end{proof}

\begin{proof}[Proof of Lemma~\ref{lem:2s}]
Let us first estimate the integral  $\int_{K}(\Om_{\cF}|_{X_t})^{\wedge k}$. Since $\Om_{\cF}$ is a positive closed $(1,1)$-current with continuous potential, it follows from~\cite[Proposition~1.11]{demailly93} that for any irreducible component $E$ of $\cX'_0$ we have
$$
(\Om_{\cF}|_{E})^{\wedge k} (Z) = 0~,
$$
so that $\mu_0(Z)=0$ where  $\mu_0 = dd^c (\log|\pi \circ p |) \wedge\Om_{\cF}^{\wedge k}$.
Since \[t \mapsto (\Om_{\cF}|_{X_t})^{\wedge k}=dd^c (\log|\pi \circ p - t|) \wedge\Om_{\cF}^{\wedge k}\] is 
continuous, for a sufficiently small compact neighborhood $K$ of $Z$ we have $\int_{K}(\Om_{\cF}|_{X_t})^{\wedge k} \le \epsilon$
for all $|t|\ll 1$.

\smallskip 

Since we argue locally and $\Phi_{\cG}$ is bounded from above, we only have to estimate the integral  $\int_{K} (\Phi_{\cG} \circ \psi)\, (\Om_{\cF}|_{X_t})^{\wedge k}$. We work in a fixed chart 
$U\ni (w_0, \cdots , w_k)$ near a point $x \in Z$ where we have
\[\Phi_\cG \circ \psi =\frac{\log r \cdot (\sum_i d_{i} \log|w_i| +  \log \max_\alpha  |h_\alpha| ) + \theta}{\sum_i b_{i} \log|w_i| + O(1)}\] where  $\theta$ is continuous, $h_\alpha$ are holomorphic,  $d_i \in \Z$, and $b_i \in \N^*$, see~\eqref{eq:estim-important} above. 
We decompose  $\Phi_\cG \circ \psi$ into the following sum $\Phi_1 + \Phi_2$, with
\[
\Phi_1= \frac{\log r \cdot (\sum_i d_{i} \log|w_i| ) + \theta}{\sum_i b_{i} \log|w_i| + O(1)} 
\text{
and }
\Phi_2= \frac{\log r \cdot (\log \max_\alpha  |h_\alpha| ) }{\sum_i b_{i} \log|w_i| + O(1)}~.
\]
Since $\Phi_1$ is bounded, we have  $\int_{K} \Phi_1 (\Om_{\cF}|_{X_t})^{\wedge k} \le \epsilon$ for $K$ and $t$ small enough by the preeceding estimate. 
Let us prove that 
\[
C_t = \int_U \log \max_\alpha  |h_\alpha| \,  (\Om_{\cF}|_{X_t})^{\wedge k} =O(1)~.
\]
Since $\int_{K} \Phi_2 (\Om_{\cF}|_{X_t})^{\wedge k} \le \frac{ C_t\cdot \log r}{\log|t|^{-1}}\to 0$, this will conclude the proof. 
Let $g$ be a continuous potential of $\Om_{\cF}$ in the open set $U$. We can then write
\[
C_t = \int_U  \log \max_\alpha  |h_\alpha| \,  (\Om_{\cF}|_{X_t})^{\wedge k}
= \int_U \log \max_\alpha  |h_\alpha|\,  \left([X_t] \wedge (dd^c)^{k} g\right)~.
\]
This integral can be now estimated using the improved Chern-Levine-Nirenberg inequalities  of~\cite[Proposition~2.6]{demailly93}
(with $u_1 =  \log \max_\alpha  |h_\alpha|$, $u_2 = \cdots = u_k = g$ and $T = [X_t]$).
Indeed since the ideal sheaf $\mathfrak{B}$ has a co-support which does not contain any vertical component, the psh function $\log \max_\alpha  |h_\alpha|$ is continuous outside 
a subvariety $W$ of $\cX'$ whose intersection with any fiber $X_t$ has codimension at least $2$ (in $\cX'$). 
\end{proof}


\section{Monge-Amp\`ere measures of uniform limits of model functions}
In this section, we show how to extend Theorem~\ref{thm:degeneration measures model} to a much larger class of measures. 
This will imply a stronger form of Theorem~\ref{thm:degeneration} from the introduction.

\subsection{Uniform limits of model functions}\label{sec:defi uniform}
We aim at proving a generalization of Theorem~\ref{thm:degeneration measures model}
to a more general class of functions than model ones. 
To that end we introduce the following definition.
\begin{definition}
A function $\varphi: X \to \R$ is said to be uniform if there exist $r>0$ and 
a sequence of regular admissible data $\cF_n$ of degree $d_n \to \infty$ such that 
\begin{equation}\label{eq:def-uniform}
\sup_{X_t}\left|\frac1{d_n} \varphi_{\cF_n} - \varphi \right| \le \epsilon_n \log|t|^{-1}
\end{equation}
for all $0<|t|\le r$ and 
for a sequence $\epsilon_n \to 0$.
\end{definition}
The condition imposed by~\eqref{eq:def-uniform} is empty outside $\pi^{-1}(\bar{\D}^*_r)$.
This causes no harm since we shall only be interested in the behaviour of uniform functions near the central fiber.

\smallskip

Observe that for a regular admissible datum $\cF$ of degree $d$, the model function $\frac1d \varphi_\cF$ 
is uniform since $\frac1{d^n} \varphi_{\cF^{\otimes n}} =\varphi_{\cF}$ for all $n$. We refer to the next section for 
more examples. 

\begin{remark}
Pick any uniform function $\varphi$ as in the definition, and consider the function $\Phi := \sn \cdot \varphi \circ \psi^{-1}$ on $\pi_{\hyb}^{-1}(\tau(\bar{\D}^*_r))$
in the hybrid space. Then~\eqref{eq:def-uniform} implies the uniform convergence $\frac1{d_n}\Phi_{\cF_n} \to \Phi$ hence $\Phi$ extends continuously to $X_{\hyb}$. Heuristically uniform functions correspond to continuous $\om$-psh function on the hybrid space, see Question~\ref{qst1} below for a conjectural characterization of uniform functions in this vein.
\end{remark}

Let us explain now how to associate a family of positive Borel measures to a uniform function.

\smallskip

 \begin{theorem}\label{thm:deg-uniform}
 Let $\varphi$ be any uniform function on $X$, and let $\cF_n$ be a sequence of regular admissible data of degree $d_n \to \infty$ 
 such that~\eqref{eq:def-uniform} holds. 
 
 For any $t\in \bar{\D}_r$, the sequence of measures $\frac1{d_n^k}\mu_{t,\cF_n,\hyb}$ converges to a positive Borel measure 
 $\MA_{t,\hyb}(\varphi)$, and we have the following weak convergence of measures
 \begin{equation}\label{eq:key-result}
\lim_{t\to0} \MA_{t,\hyb}(\varphi) = \MA_{0,\hyb}(\varphi)
 \end{equation}
 in the hybrid space $X_{\hyb}$. More precisely, for any  admissible datum $\cG$, we have
 \begin{equation}\label{eq:supreme limit}
\lim_{t\to0} \int \Phi_{\cG} \, d\MA_{t,\hyb}(\varphi) =  \int \Phi_{\cG} \, d \MA_{0,\hyb}(\varphi)
 \end{equation}

 \end{theorem}

\begin{proof}
Let $\cF_n$ be a sequence of admissible data of degree $d_n$,  such that  $\frac1{d_n} \varphi_{\cF_n}\to \varphi$, and 
$$\sup_{X_t} \left|\frac1{d_n} \varphi_{\cF_n} - \varphi\right| \le \epsilon_n \log|t|^{-1}$$ 
with $\epsilon_n \to0$. Recall that we wrote $\om$ for the curvature form of the smooth positive metrization $|\cdot|_\star$ on $\cL$, and 
$\om_t = \om|_{X_t}$.

\smallskip

For any fixed $t\in \bar{\D}_r^*$, the restriction $\varphi|_{X_t}$ is the uniform limit of the sequence of continuous functions $\frac1{d_n} \varphi_{\cF_n}|_{X_t}$, 
and  $\om_t + \frac1{d_n} dd^c \varphi_{\cF_n}|_{X_t}$ is a positive closed $(1,1)$-current for all $n\in\N$. It follows from~\cite[Corollary~1.6]{demailly93} that
$\om_t + dd^c \varphi|_{X_t}$ is also a positive closed $(1,1)$-current whose $k$-th exterior power is well-defined and 
\begin{multline*}
\mu_{n,t} := \frac1{d_n^k}\mu_{t,\cF_n,\hyb}
= 
\psi_* \left(\om_t + \frac1{d_n} dd^c  \varphi_{\cF_n}|_{X_t}\right)^{\wedge k}\\
\mathop{\xrightarrow{\hspace*{1cm}}}\limits^{n\to\infty}
\psi_* \left(\om_t + dd^c \varphi |_{X_t}\right)^{\wedge k} =:\MA_{t,\hyb}(\varphi)~.
\end{multline*}
Recall that $\sn = \frac{\mathopen| \log r \mathclose|}{\log|\pi|^{-1}}$ on $\cX$ so that  the function $\Phi = \sn \cdot \varphi \circ \psi^{-1}$ which is defined on $\psi(X)\subset X_{\hyb}$
satisfies
$$
\left|\Phi - \frac1{d_n} \Phi_{\cF_n}\right| \le \epsilon_n \mathopen| \log r \mathclose| \text{ on } \pi_{\hyb}^{-1}(\bar{\D}^*_r)~.
$$
We thus conclude that $\Phi$ extends continuously to $X_{\hyb}$ and is a uniform limit of the sequence of model functions$\frac1{d_n} \Phi_{\cF_n}$ on $X_{\hyb}$. In particular,  $g := \Phi|_{X^{\an}_{\C((t))}}$ is a uniform limit of the sequence of  model functions $\frac1{d_n} g_{\cF_n}$. It follows from Theorem~\ref{thm:key result} that the Monge-Amp\`ere measure
$\MA_{\cL}(g)$ is well-defined, and we have the weak convergence of measures
\begin{equation*}
\mu_n := \frac1{d_n^k} \mu_{0,\cF_n,\hyb} = 
\MA_{\cL}\left(\frac1{d_n}g_{\cF_n}\right)
\mathop{\xrightarrow{\hspace*{1cm}}}\limits^{n\to\infty}\,
\MA_{\cL}(g)=:\MA_{0,\hyb}(\varphi)~.
\end{equation*}
It remains to prove~\eqref{eq:supreme limit} (which implies~\eqref{eq:key-result}).

\smallskip

We claim that for any  admissible data $\cG$ there exists a constant $C(\cG) >0$ such that
\begin{equation}\label{eq:uniform estimates}
\left|\int \Phi_\cG \, d\mu_{n,t} - \int \Phi_\cG \,d\MA_{t,\hyb}(\varphi)\right| \le C(\cG) \epsilon_n 
\end{equation}
for all $t\in \bar{\D}_r^*$ and all $n$.
Indeed, using the positivity of the current $ dd^c \Phi_{\cG} + \deg(\cG)\,\om$ on $X_r$ by~\eqref{eq:local struct}, we get
\begin{multline*}
\int \Phi_\cG\, d \mu_{n,t} - \int \Phi_\cG\, d\MA_{t,\hyb}(\varphi)
=\\
\frac{\mathopen| \log r \mathclose|}{\log|t|^{-1}}\int_{X_t} \varphi_\cG \left(\om_t + \frac1{d_n} dd^c \varphi_{n}|_{X_t}\right)^{\wedge k} - 
\int_{X_t} \varphi_\cG \left(\om_t + dd^c\varphi|_{X_t}\right)^{\wedge k}
=\\
\frac{\mathopen| \log r \mathclose|}{\log|t|^{-1}}\sum_{j=0}^{k-1}
\int_{X_t} \left(\frac1{d_n} \varphi_{n} - \varphi\right) \left(\om_t + \frac1{d_n} dd^c \varphi_{n}|_{X_t}\right)^{\wedge j} \wedge \left(\om_t + dd^c\varphi|_{X_t}\right)^{\wedge (k-j-1)} \wedge  dd^c \varphi_\cG
\\
\le \frac{\mathopen| \log r \mathclose|}{\log|t|^{-1}}
\, \sup_{X_t} \left|\frac1{d_n}\varphi_{n} - \varphi \right| \times 2k\, \deg(\cG)
\end{multline*}
which implies~\eqref{eq:uniform estimates} with $C(\cG) =2k\, \mathopen| \log r \mathclose| \, \deg(\cG)$. 

\medskip

Let us now prove that $\int_{X_t}\Phi_{\cG} \,d\MA_{t,\hyb}(\varphi) \to \int_{X^{\an}_{\C((t))}}\!\!\!\! \Phi_{\cG} \,d\MA_{0,\hyb}(\varphi)$. 
To that end  we fix $\epsilon>0$ arbitrarily small, and take $n$ sufficiently large such that $\epsilon_n \le \epsilon$. 
Since  $\mu_{n,t} \to \mu_n$ as $t\to 0$ by Theorem~\ref{thm:degeneration measures model}, there exists 
$\epsilon' >0$ such that 
$$
\left|\int \Phi_{\cG} \,d\mu_{n,t} - \int \Phi_{\cG} \,d\mu_n\right|\le \epsilon$$
for all $0<|t|\le \epsilon'$. 
By~\eqref{eq:uniform estimates}, we infer 
$$ \left|\int \Phi_{\cG} \,d\MA_{t,\hyb}(\varphi) - \int \Phi_{\cG} \,d\mu_n\right|\le \epsilon (1+C(\cG)) $$
and letting $n\to\infty$ we conclude that 
$$ 
\left|\int \Phi_{\cG} \,d\mu_{t} - \int\Phi_{\cG} \,d\MA_{0,\hyb}(\varphi) \right|\le \epsilon (1+C(\cG))
$$
for all $|t|\le \epsilon'$
as was to be shown.
\end{proof}

\subsection{Example of uniform functions}\label{sec:ex uniform}
This section is logically not necessary for the rest of the paper. Recall that $|\cdot|_\star$ is a reference positively curved and smooth metric on $\cL$. 
\begin{proposition}
Let $\varphi : \cX \to \R$ be any continuous function such that 
$|\cdot|_\star e^{-\varphi}$ induces a semi-positive metric on $\cL$.
Then one can find a sequence $(\cF_n)$ of admissible data  of degree $n$ such that 
\begin{equation}\label{eq:uniform-example}
\sup_{\pi^{-1}(\bar{\D}_r)} \left|\frac1{n} \varphi_{\cF_n} - \varphi\right| \to 0 \text{  as } n\to\infty~.
\end{equation}
In particular, the function $\varphi$ is uniform.
\end{proposition}

This result shows that uniform functions form a quite large class. 
Observe that on the other hand, it is quite easy to show that $\MA_{t,\hyb}(\varphi) \to \MA_{0,\hyb}(0)$ as $t\to0$ for any functions as in the statement of the previous proposition (without
approximating by model functions).  
In fact one has the following
\begin{remark}
Suppose $|\cdot|_\star e^{-\varphi}$ is a semi-positive metric on $\cL$
that is continuous \emph{in restriction to $X$}, and such that   $\sup_{X_t} |\varphi| = o(\log|t|^{-1})$. 
Then the proof of~\eqref{eq:uniform estimates} yields
$$
\left|\int \Phi_{\cG}\,  d\MA_{t,\hyb}(\varphi) - \int \Phi_{\cG}\,  d\MA_{t,\hyb}(0)\right|
\le \frac{\mathopen| \log r \mathclose|}{\log |t|^{-1}} \sup_{X_t}|\varphi| \, 2k \deg(\cG)~,
$$
for all admissible data $\cG$, so that in particular one has $\MA_{t,\hyb}(\varphi) \to \MA_{0,\hyb}(0)$.
\end{remark}

\begin{proof}
The proof is a simple adaptation of the approximation result of Demailly, an account of which is given in~\cite[Theorem~14.21]{analytic methods}. 
For any integer $m$ write $|\cdot| = |\cdot|_\star\, e^{-\varphi}$, $|\cdot|_m = |\cdot|_\star\, e^{-m\varphi}$ (which is a metric on $\cL^{\otimes m}$). 
Let $\om$ be the curvature form of our reference metric on $\cX$, and denote by $\vol_\om = \om^{k+1}$ the volume element it defines on $\cX$. 
Recall that  $\cX_r = \pi^{-1}(\D_r)$, and $\bar{\cX}_r = \pi^{-1}(\bar{\D}_r)$.

Consider the Hilbert space \[\cH_m = \left\{ \sigma \in H^0(\cX_r, \cL^{\otimes m}), \, \int |\sigma|_m^2  \, d\vol_\om< \infty\right\}\] and set
$\varphi_m = \sup_{\sigma \in \cH_m(1)} \frac1m \log |\sigma|_\star$ where $\cH_m(1)$ is the unit ball of $\cH_m$.

We cover $\bar{\cX}_r$ by finitely many charts $U_i$ in which both $K_{\cX}$ and  $\cL$ are trivialized.
Pick any section $\sigma$ of $\cL^{\otimes m}$ over $\cX_r$. 
In each trivializing chart $U_i$ , $\sigma$ gives rise to a holomorphic function $\sigma_i$. 
We have $|\sigma|_\star = |\sigma_i| e^{-m v_i}$ for some smooth psh functions $v_i$. 

For all $x\in \cX_r\cap U_i$, and for any $\rho$ sufficiently small,  the mean value inequality for $|\sigma_i|^2$ then implies
\begin{eqnarray*}
|\sigma(x)|_\star^2 
& = &  
e^{- 2m v_i(x)} |\sigma_i(x)|^2
\le 
\frac{e^{- 2m v_i(x)}  (k+1)!}{\pi^{k+1}\rho^{2(k+1)}}\, \int_{B(x,\rho)}|\sigma_i|^2\, d\vol
\\
& \le & 
\frac{C e^{- 2m v_i(x)} }{\rho^{2(k+1)}}\, \int_{B(x,\rho)}|\sigma|_\star^2 e^{- 2m\varphi} \times e^{2m \sup_{B(x,\rho)} \varphi}\times e^{2m \sup_{B(x,\rho)} v_i}\, d\vol_\om
\end{eqnarray*}
so that
\begin{equation}\label{eq:upper-bd}
\varphi_m(x) \le \sup_{B(x,\rho)} \varphi + \frac1{2m} \log \left(\frac{C'}{\rho^{2(k+1)}}\right) + C'' \rho
~.
\end{equation}
For the lower bound, for any $x\in\cX_r$ in the chart $U_i$, one produces using Ohsawa-Takegoshi's theorem a holomorphic function $f$ such that 
$f(x)=a$ and 
$$
\int_{U_i} |f|^2 e^{-2m\varphi} \le C|a|^2 e^{-2m\varphi(x)}~.
$$ 
We abuse notation and denote again by $|\cdot|_\star$ the induced metric on $K_\cX^{\pm 1} \otimes \cL^{\otimes m}$ by our reference metric $\om$ on $\cX$ and $|\cdot|_\star$ and $\cL$.

Pick $m_0\in\N^*$ a sufficiently large integer such that $\cL^{\otimes m_0} \otimes K_{\cX}$ and $\cL^{\otimes m_0} \otimes K_{\cX}^{-1}$ are globally generated over a neighborhood of $\cX_r$. Choose two sections $\tau$ and $\tau'$ respectively of $\cL^{\otimes m_0} \otimes K_{\cX}$ and $\cL^{\otimes m_0} \otimes K_{\cX}^{-1}$ such that 
$|\tau(x)|_\star = |\tau'(x)|_\star = 1$ . Pick $\theta$ a smooth function having compact support in $U_i$ with constant value $1$ in a neighborhood of $x$. Interpret the $(0,1)$ form $\overline{\partial} (\theta f)$ as a section of  $\bigwedge^{0,1} T^* \cX_r \otimes \cL^{\otimes m}$ in the trivialization chart $U_i$, and consider the section $F= \overline{\partial} (\theta f)\wedge \tau$ of the line bundle 
$\bigwedge^{n,1} T^* \cX_r \otimes  \cL^{\otimes (m+m_0)}$.

On the line bundle $\cL^{\otimes m}\otimes \cL^{\otimes m_0}$ put the product metric $|\cdot|'_{m}$ induced by $|\cdot|_\star \frac{e^{-m\varphi}}{|x-a|^{\theta(x) (k+2)}}$ in the first factor and $|\cdot|_\star$ in the second.  The curvature form of this metric is equal to $$(m+m_0) \om + m\, dd^c \varphi  + dd^c (\theta  \log|x-a|)  \ge \om~,$$
for $m_0$ large enough.
We can solve the equation  $\overline{\partial} G = F$ where  $G$ is a section over $\cX_r$ of the line bundle  $\bigwedge^{n,0} T^*\cX_r \otimes \cL^{\otimes (m+m_0)}$ with $L^2$-norm bounded by the $L^2$-norm of $F$, see e.g.~\cite[Corollary~5.3]{demailly-cetraro} (observe that $\cX_r$ is indeed weakly pseudoconvex).
Observe that $\overline{\partial} G =0$ in a neighborhood of $x$ so that  $G(x)$ is controlled by  the $L^2$-norm of $G$ (hence of $F$) by the mean value inequality. 
We may thus replace $G$ by $G - G(x)$ and assume $G(x) =0$.

Then $\sigma = \tau' \otimes ( (\theta f)\, \tau - G)$ is a holomorphic section of the line bundle 
$\cL^{\otimes (m+2m_0)}$ such that $ \sigma(x) = \tau'(x)\otimes a \tau(x)$, and we have the integral bound
\begin{eqnarray*}
\int |\sigma|_m^2 \, d\vol_\om 
&\le&
C_1\int (|\sigma|'_m)^2 e^{-2m_0 \varphi} |x-a|^{2\theta(x) (k+2)}\, d\vol_\om 
\\
&\le& C_2 \int (|\sigma|'_m)^2\, d\vol_\om 
\le C_3 |a|^2 e^{-2m\varphi(x)}~.
\end{eqnarray*}
Choosing $a$ such that the right hand side is equal to $1$,  we obtain the lower bound
$$
\varphi_m \ge \varphi - \frac{C}{2m}~.
$$
Now fix $\epsilon >0$, and observe that $\cH_m$ is a separable Hilbert space. We can thus find finitely many sections $\sigma_0, \ldots, \sigma_l$ of $\cL^{\otimes (m+ 2m_0)}$
such that $|\varphi_m - \frac1m \log \max \{  |\sigma_0|, \ldots ,   |\sigma_l |\}|\le \epsilon$ on $\bar{\cX}_{r'}$ for some fixed $r' < r$. 

Since $\varphi$ is continuous, one may on the other hand find $\rho>0$ small enough such that $\sup_{B(x,\rho)} \varphi \le \varphi(x) + C'' \rho \le \epsilon$ for all $x \in \bar{\cX}_{r'}$. For $m$ large enough, we then obtain
$$
\left|\varphi - \frac1{m+2m_0} \log \max \{  |\sigma_0|, \ldots ,   |\sigma_l |\}\right| \le  \frac{C'''}{2m} + \epsilon
~,$$
on $\bar{\cX}_{r'}$.
This concludes the proof since $\frac1{m+2m_0} \log \max \{  |\sigma_0|, \ldots ,   |\sigma_l |\}$ is a function associated to an admissible datum of degree $m+ 2m_0$.
\end{proof}

\subsection{Degeneration of measures of maximal entropy}\label{sec:mero fam of end}
Let us now explain how the results of  Section \ref{sec:defi uniform} imply Theorem~\ref{thm:degeneration} from the introduction. 

Recall the setting. We let $R_t$ be a meromorphic family of endomorphisms of $\PP^k_\C$ of a fixed degree $d$ parameterized by the unit disk. 
In other words, we suppose given $k+1$ homogeneous polynomials $P_{0,t}(w_0, \ldots, w_k), \ldots , P_{k,t}(w_0, \ldots, w_k)$
of degree $d$ whose coefficients are meromorphic functions on $\D$ with a single pole at the origin. 
These polynomials are uniquely determined up to the multiplication by a meromorphic function $h(t)$ in $\D$.

We also assume that for any $t\in\D^*$
these polynomials have no common zeroes so that the map
$$
R_t([w]) = R_t([w_0: \cdots: w_k]) = [P_{0,t}(w): \cdots: P_{k,t}(w)] 
$$
has no indeterminacy point. For any integer $n$, we shall write
$$
R^{\circ n}_t([w]) =  [P^n_{0,t}(w): \cdots: P^n_{k,t}(w)]~. 
$$
Recall that each polynomial $P^n_{i}$ defines a meromorphic section of the line bundle 
$\cL^{\otimes d^n}$ on $\cX := \PP^k_\C \times \D$ where $\cL := \pi_1^* \cO_{\PP^k_\C}(1)$ and $\pi_1$ denotes
 the first projection. We endow $\cL$ with the pull-back of the metric on $\cO_{\PP^k_\C}(1)$ whose curvature form is the standard Fubini-Study K\"ahler
 metric.
 
 \smallskip

 Since the polynomials $P^n_{i,t}$ have no common zeroes over the $ \PP^k_\C \times \{ t \}$ when $t\in \D^*$, 
 the fractional ideal sheaf  $\mathfrak{A}_n := \langle P^n_{t,0}, \ldots, P^n_{t,k}\rangle$ is vertical. Let $p_n: \cX_n \to \cX$ 
be any log-resolution of this vertical ideal sheaf. The set of sections  $\{P^n_{t,0}, \ldots, P^n_{t,k}\}$ and the degeneration $\cX_n$  defines a regular admissible datum $\cF_n$ of degree $d^n$ by Proposition~\ref{prop:equiv model X}, whose model function on $X$ is given by 
$$
\varphi_{\cF_n} = \log \left( 
\frac{\max \{|P^n_{0,t}(w)|, \ldots , |P^n_{k,t}(w)|\}} {(|w_0|^{2} + \cdots + |w_k|^2)^{d^n/2}}
\right)~.
$$
The key estimate is given by the following (standard) result:
\begin{proposition}\label{prop:key-estim}
There exists a positive constant $C>0$ such that 
\begin{equation}\label{eq:demarco}
\left|\frac1{d^{n+1}} \varphi_{\cF_{n+1}} - \frac1{d^n}\varphi_{\cF_n}\right| \le \frac{C\, \log|t|^{-1}}{d^n}
\end{equation}
on $\PP^k_\C\times\bar{\D}^*_r$.
\end{proposition}
\begin{proof}[Proof of Proposition~\ref{prop:key-estim}]
Observe that $P^{n+1}_{i,t} = P_{i,t}(P^n_{0,t}, \cdots, P^n_{k,t})$ for all $n$ so that~\eqref{eq:demarco}
is a consequence of the bound
$$
c\, |t|^{M} \le   \frac{\max \{|P_{0,t}(w)|, \cdots , |P_{k,t}(w)|\}} {\max \{|w_0|^d, \cdots , |w_k|^d\}}  \le C\, |t|^{-N} ~.
$$
for some $c, C>0$, and $M,N \in \N^*$. By compactness of $\bar{\D}_r$, it is sufficient to get this bound in a neighborhood of the origin.

The upper bound is easy to obtain since  $|P_{i,t}(w)| \le C |t|^N \max \{|w_0|^d, \cdots , |w_k|^d\}$ for any $i$.
The lower bound follows from the Nullstellensatz applied in the algebraic closure $\widehat{\cM}$ of the field $\cM$ of meromorphic functions on $\D^*$.
Observe that any element $g\in \widehat{\cM}$ can be represented (non uniquely) by a Puiseux series converging in some neighborhood of the origin, 
so that there exists a rational number $q$ and a positive constant such that  $|g(t)| \le C |t|^q$ for all $t$ small enough.

Since the polynomial $P_{i,t}$ have no common factors for all $t\in \D^*$, it follows that the subvariety of $\A^{k+1}_{\widehat{\cM}}$ defined by the vanishing of these polynomials is reduced to the origin. We may thus find an integer $N$ and homogeneous polynomials $q_{i,j,t}$ of degree $N-d$ with coefficients in $\widehat{\cM}$ such that
$$
w_i^N = \sum_j q_{i,j,t} P_{j,t}~.
$$
Assuming $|q_{i,j,t}(t)| \le C |t|^q$ near $0$ for all $i,j$ and taking norms  of both sides, we get
$$
\max \{|w_0|, \cdots , |w_k|\}^N \le C |t|^q \times \max \{|w_0|, \cdots , |w_k|\}^{N-d} \times \max \{|P_{0,t}(w)|, \cdots , |P_{k,t}(w)|\}~,
$$
which implies the lower bound.
\end{proof}

Proposition~\ref{prop:key-estim} implies that $\varphi_R := \lim_{n\to\infty} \frac1{d^n} \varphi_{\cF_n}$ is a well-defined function on $\PP^k_\C\times \D^*$
which is uniform in the sense of \S\ref{sec:defi uniform}.

\begin{proof}[Proof of Theorem~\ref{thm:degeneration}]
By Theorem~\ref{thm:deg-uniform}, we have the convergence of measures
$\MA_{t,\hyb}(\varphi_R) \to \MA_{0,\hyb}(\varphi_R)$ in the hybrid space associated to $\PP^k_\C\times \D$.

To conclude the proof it remains to relate $\MA_{t,\hyb}(\varphi_R)$ to the measure of maximal entropy $\mu_t$ of the endomorphism $R_t$, 
and $\MA_{0,\hyb}(\varphi_R)$ to the Chambert-Loir measure of the dynamical system $\mathcal{R}$ induced by the family $\{R_t\}$ on 
$\PP^k_{\C((t))}$.

\medskip

Let $\om_{\FS}$ be the standard Fubini-Study $(1,1)$-form on $\PP^k_\C$ so that
$$
\frac1d R_t^* \om_{\FS} - \om_{\FS} = \frac12 dd^c \log \left(
\frac{|P_{0,t}(w)|^2+  \cdots + |P_{k,t}(w)|^2} {|w_0|^{2} + \cdots + |w_k|^2}
\right)~.
$$
It follows from the previous proposition  that
$\frac1{d^n} (R^n_t)^* \om_{\FS}$ converges to a positive closed $(1,1)$-current $T$
with continuous potential, and~\cite[Th\'eor\`eme 3.3.2]{sibony} and~\cite{briend-duval} implies that  $T^{\wedge k}$ is the unique measure of maximal entropy of $R_t$ hence is equal to $\mu_t$.
On the other hand for each $n$, we have
\begin{multline*}
\frac1{d^n} (R^n_t)^* \om_{\FS} - \left(\om_{\FS} + \frac1{d^n}dd^c \varphi_{\cF_n}|_{\PP^k_\C\times\{t\}}\right) 
= \\
\frac1{d^n} \, dd^c 
\log \left(
\frac{(|P^n_{0,t}(w)|^2+  \cdots + |P^n_{k,t}(w)|^2)^{1/2}}{\max \{|P^n_{0,t}(w)|, \cdots , |P^n_{k,t}(w)|\}}
\right)~.
\end{multline*}
The right hand side is the $dd^c$ of  a function with values in $[0, \frac{\log (k+1)}{2d^n}]$, and therefore we conclude that 
$T = \lim_n (\om_{\FS} + \frac1{d^n}dd^c \varphi_{\cF_n}|_{\PP^k_\C\times\{t\}})$, and 
$T^{\wedge k} =  \lim_n (\om_{\FS} + \frac1{d^n}dd^c \varphi_{\cF_n}|_{\PP^k_\C\times\{t\}})^{\wedge k}$.
Unwinding definitions, we see that the latter convergence implies $\psi_*(\mu_t) = \MA_{t,\hyb}(\varphi_R)$ in the hybrid space.

\medskip

To identify $\MA_{0,\hyb}(\varphi_R)$ with the Chambert-Loir measure of $\mathcal{R}$, we proceed as follows.
By definition $ \MA_{0,\hyb}(\varphi)=  \MA_{\cL}(g)$ where $g = \lim_{n\to\infty} \frac1{d^n} g_{\cF_n}$. 
We claim that $g =g_\mathcal{R}$ as defined in~\eqref{eq:gR} hence $ \MA_{0,\hyb}(\varphi)= \mu_\mathcal{R}$ which proves the theorem.

Let $P_i$ be the homogeneous polynomial of degree $d$ and coefficients in $\C((t))$ associated to $P_{i,t}$.
Observe first that we have the following identity in $\PP^k_\C\times \D$:
\begin{equation}
\frac1{d^n} \varphi_{\cF_n} = \frac1{d^n} \varphi_{\cF_1} + \sum_{j=1}^{n-1}\left(\frac1{d^{j+1}} \varphi_{\cF_{j+1}} - \frac1{d^j} \varphi_{\cF_j}\right)
=\frac1{d^n} \varphi_{\cF_1} +  \sum_{j=0}^{n-1}\frac1{d^{j}} \, \tilde{\varphi} \circ R^n
\end{equation}
where 
$$
\tilde{\varphi}([w],t)= \log \left(\frac{\max\{|P_{0,t}|, \cdots, |P_{k,t}|\}}{\max\{|w_0|, \cdots, |w_k|\}^d}\right)~.
$$
Recall the definition of $g_1$ in \S\ref{sec:ACL measure}, and observe that this function equals $g_{\cF_1}$ by definition.

\begin{lemma}\label{lem:final estimate}
The function $\sn \cdot \tilde{\varphi} \circ \psi$ extends continuously to the hybrid space and its restriction to $\pi^{-1}(\tau(0))$ is equal to $g_1$.
\end{lemma}
From this lemma and Theorem~\ref{thm:extension} we get 
$$\frac1{d^n} g_{\cF_n} = \frac1{d^n} g_{\cF_1} + \sum_{j=0}^{n-1}\frac1{d^{j}}  g_1 \circ \mathcal{R}^j$$
and letting $n\to\infty$, we conclude that $g= g_\mathcal{R}$ by~\eqref{eq:gR}.
\end{proof}

\begin{proof}[Proof of Lemma~\ref{lem:final estimate}]
One has  $\sn \cdot (\tilde{\varphi} - \varphi_{\cF_1}) = \sn \cdot \log \left(\frac{(|w_0|^2+ \cdots+ |w_k|^2)^{d/2}}{\max\{|w_0|, \cdots, |w_k|\}^d}\right)$
 so that this function extends continuously to the hybrid space with constant value $0$ on $\pi_{\hyb}^{-1}(\tau(0))$.
\end{proof}

\subsection{Lyapunov exponents of endomorphisms}\label{sec:lyap}

Let $R = [P_0: \cdots : P_k]$ be an endomorphism of the projective complex space $\PP^k_\C$ given in homogeneous coordinates by $k+1$ polynomials
of degree $d$. The norm of the determinant of the differential $ \mathopen\| \det (dR)\mathclose\|$
computed with respect to  Fubini-Study K\"ahler form $\om$ satisfies  
$R^* (\om^{\wedge k}) = \mathopen\| \det (dR)\mathclose\|^2 \, \om^{\wedge k}$  and  a direct computation in homogeneous coordinates shows: 
\[\mathopen \| \det (dR)\mathclose\| =\frac1{d}
\left|\det 
\left[
\frac{\partial P_i}{\partial w_j}
\right]_{i,j}\right|
\times \,\left( \frac{|w_0|^2+ \cdots + |w_k|^2}{|P_0|^2+ \cdots + |P_k|^2} \right)^{k/2}~,
\]
see~\cite[Lemma 3.1]{bedford-jonsson}.
Recall that the sum of the Lyapunov exponents of $R$ is given by the formula:
$$
\Lyap(R) = \int \log\mathopen\| \det (dR)\mathclose\|\, d\mu_R~,
$$
where $\mu_R$ is the measure of maximal entropy of $R$. Observe also that $\log\mathopen \|\det (dR)\mathclose\|$ is locally the sum of a psh function and a smooth function so that the integral is converging since $\mu_R$ is locally the Monge-Amp\`ere measure of a continuous function. 
It was proved in~\cite{BD-Lyap} that $\Lyap(R) \ge \frac{k}2 \log d$.

\smallskip

For an endomorphism $\mathcal{R}=[P_0: \cdots : P_k]$ defined over $\PP^k_{\C((t))}$, then one uses a slightly different formula setting 
$$ \| \det (d\mathcal{R})\| = 
\left|\det 
\left[
\frac{\partial \sP_i}{\partial w_j}
\right]_{i,j}\right|
\times \, \left(\frac{\max \{ |w_0|,  \ldots , |w_k|\}}{\max \{ |P_0|,  \ldots , |P_k|\}}\right)^2~,
$$
compare with~\cite[(3.1)]{oku-repelling}. The sum of the Lyapunov exponents\footnote{One can define each individual Lyapunov exponent of $\mathcal{R}$ by looking at the limits
$\frac1{n}\int \log \| \bigwedge^l (d\mathcal{R}^n)\|\, d\mu_\mathcal{R}$ as $n\to\infty$ for $l\in \{ 1, \cdots, k\}$ which exist by Kingman's theorem.} of the Chambert-Loir measure of $\mathcal{R}$ is defined analogously to the complex case by the formula:
$$
\Lyap(\mathcal{R}) = \int \log\mathopen\| \det (d\mathcal{R})\mathclose\|\, d\mu_\mathcal{R}~.
$$
This integral makes sense and is finite by Theorem~\ref{thm:key result}.

\begin{proof}[Proof of Theorem~\ref{thm:Lyapunov}]
Introduce the two (possibly singular) admissible data $\cG_1$ and $\cG_2$ corresponding to the section $\det 
\left[
\frac{\partial P_{i,t}}{\partial w_j}
\right]_{i,j}$ and to the family of sections $P_{0,t},  \cdots , P_{k,t}$ respectively. They are of degree $(2d-2)$ and $d$ respectively.
We have 
$$
 \varphi_{\cG_1} =\left|\det 
\left[
\frac{\partial P_{i,t}}{\partial w_j}
\right]_{i,j}\right|
\times \, \frac1{\max \{ |w_0|, \cdots,|w_k|\}^{2d-2}}~,
\text{ and }
\varphi_{\cG_2} = \frac{\max \{ |P_{0,t}|, \cdots , |P_{k,t}|\}}{\max\{ |w_0|, \cdots , |w_k|\}^d}~,
$$
so that 
$$\log \mathopen\| \det(R_t)\mathclose \|  = \int \left(\varphi_{\cG_1} - 2 \varphi_{\cG_2} + \tilde{\varphi}\right)\, d\mu_{R_t}$$
where 
$$
\tilde{\varphi} = 2 \log \left(\frac{\max \{ |P_{0,t}|, \ldots , |P_{k,t}|\}}{|P_{0,t}|^2+ \cdots + |P_{k,t}|^2}\right) - 
\log \left(\frac{\max \{ |w_0|, \ldots,|w_k|\}}{|w_0|^2+ \cdots + |w_k|^2}\right)~,
$$is a bounded function on $\PP^k_\C \times \D^*$.
We now apply Theorem~\ref{thm:deg-uniform} to the uniform function $\varphi_R$, and we get the series of equalities
\begin{eqnarray*}
\Lyap(R_t)
&=&
 \int  (\varphi_{\cG_1} - 2 \varphi_{\cG_2} + \tilde{\varphi} ) \, d\mu_{R_t}
\\
&=& \frac{\log|t|^{-1}}{\mathopen| \log r \mathclose|} \, \int  (\Phi_{\cG_1} - 2 \Phi_{\cG_2}) \, d\MA_{t,\hyb}(\varphi_R) + O(1)
\\
&=& 
\frac{\log|t|^{-1}}{\mathopen| \log r \mathclose|} \, \int  (g_{\cG_1} - 2 g_{\cG_2}) \, d\MA_{0,\hyb}(\varphi_R) + o(\log|t|^{-1})
\\
&=& 
\Lyap(\mathcal{R}) \, \frac{\log|t|^{-1}}{\mathopen| \log r \mathclose|}+ o(\log|t|^{-1})~.
\end{eqnarray*}
This concludes the proof.
\end{proof}


\section{Questions}\label{sec:questions}

\subsection{Characterization of uniform functions}
The notion of  uniform function a priori depends on the choice of a smooth positive metrization on $\cL$. It would be interesting to explore if
one can give a more intrinsic definition of uniform functions not relying on the existence of an approximating sequence of model functions. 

Let $T$ be any positive closed $(1,1)$ current on $\cX$, and let $E$ be any irreducible component of the central fiber of an snc model $p: \cX' \to \cX$.
Then we set $g_T(x_E)$ to be the quotient of the Lelong number of $T$ at a general point in $E$ divided by the integer $b_E = \ord( p^* \pi^* t)$.
\begin{qst}\label{qst1}
A function $\varphi: X \to \R$ is uniform iff
\begin{itemize}
\item
it is continuous in a neighborhood of $\bar{X}_r$;
\item
there exists a positive closed $(1,1)$ current $T$ on $\cX_r$ such that $T|_{X_r} = \om + dd^c \varphi$;
\item
the function $g_T$ extends continuously to  $X^{\an}_{\C((t))}$.
\end{itemize}
\end{qst}
The forward implication is easy. 

\medskip

It was proved in~\cite{MA} (see also~\cite{burgos et al}) that one can solve the Monge-Amp\`ere equation $\MA_{\cL}(g) = \mu$ 
for a suitable class of positive measures $\mu$ on $X^{\an}_{\C((t))}$.
\begin{qst}
Let $|\cdot|_{\cL}\, e^{-g}$ be any continuous semi-positive metrization of $L^{\an}_{\C((t))}$. 
Is it possible to find a uniform function $\varphi$ such that $ \MA_{\cL}(g) = \MA_{0,\hyb}(\varphi)$?
\end{qst}

\subsection{Controlling the error term}

Let us first sketch the proof of the following
\begin{theorem}
Suppose $\dim(X) =2$, and let $\cF$ and $\cG$ be two admissible data with $\cF$ being regular. Then the error function
$$
\cE(t) := \int \varphi_{\cG} \, d( \mu_{\cF,t}) - \left(\int g_{\cG} \, d(\mu_{\cF,\NA})\right) \, \frac{\log\mathopen|t\mathclose|^{-1}}{\mathopen| \log r \mathclose|} 
$$
extends continuously through the origin. 
\end{theorem}
\begin{proof}
Restricting the situation to a smaller disk if necessary, we may choose a snc model $p: \cX'\to \cX$ that is a log-resolution of both fractional ideal sheaves
associated to $\cF$ and $\cG$. Note this step is only possible when $\dim(X) =2$ since $\cG$ may be singular.

Recall that $\mu_{\cF,t} = (\om + dd^c \varphi_\cF)|_{X_t}$
where $ \varphi_\cF = \log \max_i \{ |\tau_i|_\star\}$ and $\tau_0, \ldots, \tau_l$ are the sections defining $\cF$, see~\eqref{eq:defvarphi}.
Fix any large integer $n$ and define the real-analytic function 
$$ \varphi_{\cF,n} = \frac1n \, \log\left( \sum_i |\tau_i|^n_\star \right),$$ 
so that $\sup |\varphi_{\cF,n} - \varphi_\cF| \le \frac{\log(l+1)}n$. 
By integration by parts we get 
$$\left|\int \varphi_{\cG} \, d( \mu_{\cF,t})  - \int \varphi_{\cG} \, d( \mu_{\cF,t,n}) \right| \le C \sup | \varphi_{\cF,n} -  \varphi_{\cF}| \le \frac{C'}n~,$$
for some constants $C,C'$ where $\mu_{\cF,t,n} := (\om + dd^c \varphi_{\cF,n})|_{X_t}$. 

It follows that it is only necessary to prove that for any $n$, there exists  a constant $c\ge0$ such that
$\int \varphi_{\cG} \, d( \mu_{\cF,t,n}) - c \log|t|^{-1}$ extends continuously through the origin. 
To see this we cover the central fiber by finitely many charts. Fix such a  chart, and  choose
coordinates $z_1, z_2 \in \D$ such that $ \pi \circ p (z_1,z_2) =z_1^{b_1}z_2^{b_2}$ for some integers $b_1,b_2 \in \N$. 
We observe that the proof of Theorem~\ref{thm:model1} applies and shows that $dd^c \varphi_{\cF,n} + p^* \om =  \Om +  [D']$
where $\Om$ is a real analytic $(1,1)$ closed positive form, and $[D']$ is a vertical divisor.

We may thus complete the proof  using the next lemma.
\end{proof}
\begin{lemma}
Let $\varphi$ be a function having compact support in the unit polydisk  and such that 
$\varphi - a_1 \log|z_1| - a_2 \log|z_2|$ is continuous for some $a_1, a_2 \in \R_+$. 
Let $\Om$ be any real-analytic closed $(1,1)$ form. Pick any pair of integers $b_1, b_2 \in \N^2$, and write
$Y_t = \{ (z_1, z_2) \in \D^2, \, z_1^{b_1} z_2^{b_2} =  t\}$. Then there exists a constant $c\ge0$ such that 
$$
\int_{Y_t} \varphi \, \Om - c \log |t|^{-1}
$$
extends continuously through the origin.
\end{lemma}

\begin{proof}
It is a theorem of Stoll~\cite{stoll} that the fiber integral $\int_{Y_t} \varphi \, \Om$ is a continuous function of $t$ when $\varphi$ is continuous. 
Under our standing assumption, this also follows from the classical Chern-Levine-Nirenberg inequality which implies $\Om|_{Y_t} \to \Om \wedge dd^c \log |z_1^{b_1} z_2^{b_2}|$. 
Following Barlet~\cite{barlet}, it is even possible to find a complete asymptotic expansion in $t^{\kappa} \bar{t}^{\kappa'} (\log|t|)^q$ of the function $t\mapsto \int_{Y_t} \varphi \, \Om$ when $\varphi$ is smooth. 
In any case, we may suppose $\varphi = \log|z_2|$. One can also assume $b_1, b_2$ are positive, since otherwise the result is easy to prove. 

As $\Om$ is a real-analytic positive closed current, we may write $\Om = dd^c (\phi)$ where $\phi$ is a real-valued real-analytic function. 
We expand it into power series
$$
\phi = \sum_{IJ} \phi_{IJ} z^I \bar{z}^J $$
where $I = (i_1, i_2)$ and $J= (j_1,j_2)$ are multi-indices and $z^I \bar{z}^J = z_1^{i_1}z_2^{i_2}\bar{z_1}^{j_1} \bar{z_2}^{j_2}$.
We may suppose $\Om$ is defined in a neighborhood of the unit polydisk so that $\sum  |\phi_{IJ}|  <\infty$. 
We now fix $t\in \D^*$, and pick $\tau \in \D^*$ such that $\tau^{b_1} =t$. We use the parameterization $h(w) = (\tau/w^{b_2}, w^{b_1})$ of $Y_t$. 
Observe that
\begin{equation}\label{eq777}
\int_{Y_t} \log|z_2|\,  i \partial \bar{\partial} (\phi)
= 
\sum_{IJ} b_1 \phi_{IJ} \, \int_{|\tau|^{1/b_2} \le |w| \le 1} \log|w|\,  i \partial \bar{\partial} (z^I \bar{z}^J \circ h) 
\end{equation}
and 
\begin{align*}
\partial \bar{\partial} (z^I \bar{z}^J \circ h) 
&= h^* \left[z^I \bar{z}^J \left( i_1j_1\frac{dz_1\wedge d\bar{z_1}}{|z_1|^2} +  i_1j_2\frac{dz_1\wedge d\bar{z_2}}{z_1\bar{z_2}} +
 i_2j_1\frac{dz_2\wedge d\bar{z_1}}{z_2\bar{z_1}} +  i_2j_2\frac{dz_2\wedge d\bar{z_2}}{|z_2|^2} \right)\right]
 \\
&=  \tau^{i_1}\bar{\tau}^{j_1} w^{-i_1b_2+i_2 b_1} \bar{w}^{-j_1b_2+j_2 b_1} \left( i_1j_1b_2^2 - i_1j_2 b_1 b_2 - 
 i_2j_1b_1b_2 +  i_2j_2b_2^2\right) \frac{dw\wedge d\bar{w}}{|w|^2} 
\end{align*}
Using polar coordinates we get 
\[
\int_{|\tau|^{1/b_2} \le |w| \le 1} \log|w|\,  i \partial \bar{\partial} (z^I \bar{z}^J \circ h)  =0
 \]
 except if $\kappa:= -i_1b_2+i_2 b_1 = -j_1b_2+j_2 b_1$, in which case we have
\[
\Delta_{IJ}(t) = \int_{|\tau|^{1/b_2} \le |w| \le 1} \log|w|\,  i \partial \bar{\partial} (z^I \bar{z}^J \circ h) 
= 2\pi\, \kappa^2 \tau^{i_1}\bar{\tau}^{j_1}\,  \int_{|\tau|^{1/b_2} \le r \le 1}r^{2\kappa -1} \log r\,  dr
 ~.\]
 Note that 
 \[
 \int r^{2\kappa -1} \log r  \, dr = \frac{r^{2\kappa}\log r }{2\kappa} - \frac{r^{2\kappa}}{4\kappa^2}
 \text{ if } \kappa \neq0.\]
  It follows that when $\kappa >0$ and $I \neq 0$, then 
 \[
 |\Delta_{IJ}(t)|\le 2(|I|+ |J|) \, |t|^{2/b_1b_2} \log |t| = o(1)~;\]
 when $\kappa <0$ and $i_2 + j_2 \neq 0$,
 then 
 \[
 |\Delta_{IJ}(t)|\le |t|^{1/b_2} \log|t| = o(1)~;\]
and  otherwise $\kappa <0$,  $i_2 = j_2 =0$, $i := i_1 =j_1$ and 
\[
\left| \Delta_{IJ}(t) - \frac{i \pi}{b_1b_2} \log |t| \right| \le |t|^{2/b_1} = o(1)~.
\]
We conclude by summing up all contributions over all multi-indices $I,J$ in~\eqref{eq777}, using the fact that $\phi$ being real we have $\phi_{IJ} = \overline{\phi_{JI}}$.
\end{proof}

The previous arguments are combinatorially more involved in higher dimensions but we think that they apply again almost verbatim.
They should also be useful to treat the case $\varphi_\cF$ is replaced by any continuous function $\varphi: X \to \R$ such that 
the metrization $|\cdot|_\star e^{-\varphi}$ is semi-positive and continuous as in \S\ref{sec:ex uniform}.

It would  be interesting to develop tools to understand when $\cE$ remains continuous when $\varphi_\cF$ is replaced by a general uniform function $\varphi$. 
Observe that when $\varphi$ is the uniform function of a degenerating family of endomorphisms then it is known that this error is unbounded in general,  but it is expected that 
it is continuous under suitable assumptions on the family (for instance when it is induced by an algebraic family defined over a number field). We refer to 
 the discussion after Conjecture~\ref{lyap conj} in the introduction for references on this problem.

\subsection{Degeneration of Monge-Amp\`ere measures in a fixed model}

Given any regular admissible datum $\cF$, and any model $\cX$ (not necessarily a resolution of $\cF$), we have already observed that the family of measures $\mu_{t,\cF}$ converges to 
a positive measure $\mu_0$ on the central fiber $\cX_0$. In a joint work with E. Di Nezza~\cite{dinezza-favre}, we prove that this measure can be decomposed as a finite sum of positive measures $\mu_0 = \sum_Z \mu_Z$ where $Z$ ranges over all irreducible subvarieties of $\cX_0$, and $\mu_Z$ is
the Monge-Amp\`ere measure of a H\"older continuous quasi-psh function defined on $Z$. It is also possible to argue that the total mass of $\mu_Z$ is equal to $ \mu_{\cF,\NA}(\red_\cX^{-1}(Z))$
where $\red_\cX: X^{\an}_{\C((t))} \to \cX_0$ is the canonical reduction map sending a point to its center. Recall that this map is anti-continuous so that  $\red_\cX^{-1}(Z)$ is an open set.

DeMarco and Faber~\cite[Theorem~B]{demarco-faber} proved that this picture remains valid 
in the case of measures of maximal entropy of a degenerating family of endomorphisms of the Riemann sphere. 
It is particularly challenging  to extend these results to families of Monge-Amp\`ere measures associated to 
degenerating family of endomorphisms of higher dimensional projective spaces, and then to 
any arbitrary uniform functions.

\subsection{Degeneration of volume forms}
It would be interesting to further investigate the relationship between the two convergence theorems of measures in the hybrid space
given by Theorem~\ref{thm:hybrid} of the present paper and~\cite[Theorem~A]{boucksom-jonsson}.
Let us recall briefly the setting of the latter paper (we have changed slightly their notation so as to match with ours).

Suppose $X \to \D^*$ is a smooth and proper submersion, and let $\pi: \cX\to \D$ be an snc model of $X$. 
To simplify the discussion we shall assume furthermore that $\cX$ is smooth and that there exists a relatively ample line bundle $\cL \to \cX$.

Let $K_{X/\D^*}$ be the relative canonical line bundle over the punctured disk: in a trivialization $(z_1, \ldots, z_k, t)$ where
$\pi(z,t) =t$, then sections of $K_{X/\D^*}$ are $k$-forms $\alpha(z,t) dz_1 \wedge \cdots \wedge dz_k$ with $\alpha$ holomorphic.
Suppose that there exists a line bundle $\cK \to \cX$ whose restriction to $X$ is equal to $K_{X/\D^*}$, and pick 
any smooth metric $h$ on $K_{X/\D}$ that extends continuously to $\cK$.

For any fixed $t\in \D^*$ one may consider the smooth volume form $\mu_t$ given locally by
$\mu_t = \frac{\Om \wedge \bar{\Om}}{|\Om|^2_h}$ where $\Om$ is any local section of $K_{X/\D^*}$. 
The family of measures $\{\mu_t\}_{t\in \D^*}$ is in fact smooth, and  S. Boucksom and M. Jonsson gave a precise asymptotic formula 
for the total mass $\mu_t(X_t)$ as $t\to 0$. They also proved that the probability measures $\nu_t = \mu_t/\mass(\mu_t)$
converge to an explicit measure $\nu_{\NA}$ in the hybrid space. 

Let us fix any smooth positive metric on $\cL$, and denote by $\om$ its curvature form. The restriction $\om_t := \om|_{X_t}$ is a K\"ahler form for any $t\in \D^*$.
Recall that $\delta = \int_{X_t} \om_t^k$ is independent on $t$.  Now for any fixed $t\in\D^*$ we may solve the Monge-Amp\`ere equation 
$(\om_t + dd^c\varphi_t)^k = \delta \, \nu_t$ and $g_t$ is uniquely determined if we normalize it by the condition $\sup_{X_t} \varphi_t =0$. 

\begin{qst}
Is it true that the family of functions $\varphi_t$ is uniform in the sense of \S\ref{sec:defi uniform}?
\end{qst}

\smallskip

If the answer to the previous question is positive, then we may consider the associated function $g$ on $X^{\an}_{\C((t))}$ which defines a continuous semi-positive metrics on $L^{\an}_{\C((t))}$ and Theorem~\ref{thm:deg-uniform} together with the results of~\cite{boucksom-jonsson} imply $\MA_{\cL}(g) = \lim_t \nu_t = \delta\, \nu_{\NA}$.



\end{document}